\documentclass[a4paper,11 pt,twoside]{amsart}
\usepackage{graphicx}
\usepackage{amssymb}
\usepackage[french, english]{babel}
\usepackage[utf8]{inputenc}
\usepackage{amsmath,amsfonts,amssymb,amsthm,stmaryrd}
\usepackage{dsfont}
\usepackage[T1]{fontenc}
\usepackage{hyperref}
\usepackage{caption} 
\usepackage{mathrsfs}
\usepackage{pgfplots}
\pgfplotsset{compat=1.15}
\usepackage{relsize}
\usepackage{tikz}
\usepackage{bbm}
\usetikzlibrary{arrows}
\usepackage{pgf}
\usepackage{appendix}
\usepackage{enumitem}
\usepackage{fancyhdr}
\usepackage{tikz}
\everymath{\displaystyle}

\newcommand{\too}{\longrightarrow}
\newcommand{\eps}{\varepsilon}

\newcommand{\R}{\mathbb R}

\newcommand{\F}{\mathcal F}

\newcommand{\E}{\mathbb E}

\renewcommand{\P}{\mathbb P}

\newcommand{\1}{\mathds 1}

\newcommand{\dx}{\mathrm{d}}

\newcommand{\Lf}{\mathcal{L}}
\newcommand{\A}{\mathcal{A}}
\newcommand{\gap}{\text{gap}}
\newcommand{\Var}{\text{Var}}

\theoremstyle{definition}
\newtheorem{defn}{Definition}

\theoremstyle{plain}

\newtheorem{thm}{Theorem}
\newtheorem{prop}[defn]{Proposition}

\newtheorem{lem}[defn]{Lemma}
\newtheorem{rem}[defn]{Remark}

\definecolor{darkgreen}{rgb}{0.1,0.7,0.1}

\definecolor{darkred}{rgb}{0.7,0.1,0.1}

\begin{document}

\title{Hydrodynamic limit and cutoff for the biased adjacent walk on the simplex}

\author{Cyril Labb\'e}
\address{Universit\'e de Paris, Laboratoire de Probabilit\'es, Statistiques et Mod\'elisation, UMR 8001, F-75205 Paris, France}
\email{clabbe@lpsm.paris}
\author{Engu\'erand Petit}
\address{Universit\'e Paris-Dauphine, PSL University, Ceremade, CNRS, 75775 Paris Cedex 16, France.}
\email{petit@ceremade.dauphine.fr}

\pagestyle{fancy}
\fancyhead[LO]{}
\fancyhead[CO]{\sc{C.~Labb\'e and E.~Petit}}
\fancyhead[RO]{}
\fancyhead[LE]{}
\fancyhead[CE]{\sc{Cutoff for the biased adjacent walk}}
\fancyhead[RE]{}

\date{}

\begin{abstract}
We investigate the asymptotic in $N$ of the mixing times of a Markov dynamics on $N-1$ ordered particles in an interval. This dynamics consists in resampling at independent Poisson times each particle according to a probability measure on the segment formed by its nearest neighbours. In the setting where the resampling probability measures are symmetric, the asymptotic of the mixing times were obtained and a cutoff phenomenon holds. In the present work, we focus on an asymmetric version of the model and we establish a cutoff phenomenon. An important part of our analysis consists in the derivation of a hydrodynamic limit, which is given by a non-linear Hamilton-Jacobi equation with degenerate boundary conditions.

\medskip

\noindent
{\bf MSC 2010 subject classifications}: Primary 60J25; Secondary 37A25, 70H20.\\
 \noindent
{\bf Keywords}: {\it Mixing time; Cutoff; Adjacent walk; Hydrodynamic limit; Hamilton-Jacobi equation}
\end{abstract}

\maketitle

\setcounter{tocdepth}{1}
\tableofcontents

\section{Introduction}

The investigation of the mixing times of (sequences of) Markov chains has given rise to a vast literature. In some situations, a \emph{cutoff phenomenon} occurs: the distance to equilibrium falls abruptly at some critical time from its maximal value to $0$. This phenomenon was introduced in the context of card shuffling by Aldous and Diaconis in the eighties~\cite{DiaSha,AldDia,diaconis1996cutoff}. Although a general theory is still missing, it has been established for a variety of discrete models~\cite{LevPerWil}.  On the other hand, there are relatively few examples of Markov processes, taking values in continuous state-spaces, for which a cutoff phenomenon is proved, see for instance~\cite{Lachaud,Meliot,Barrera,HoughJiang17,BarHogPar}. The present work focuses on a continuous state-space dynamics that presents an asymmetry: in this context, the determination of the asymptotic of the mixing times, with a sharp prefactor, requires specific information on the process notably its hydrodynamic limit.

\subsection{The model}

We consider $N-1$ ordered particles in the interval $[0,N]$ that evolve through random resampling events. More precisely, on the state-space
$$S_N:= \big\{x=(x_0,...,x_N)\in\R^{N+1}\;,\quad 0=x_0\leqslant x_1\leqslant...\leqslant x_N=N \big\}\;,$$
we are given independent rate $1$ Poisson clocks attached to each $k\in \{1,\ldots,N-1\}$. We consider the continuous-time Markov process $(X^x(t),t\ge 0)$ that starts at time $0$ from some configuration $x\in S^N$, and that evolves as follows: if the $k$-th clock rings, say at time $t$, then the $k$-th particle, $X_k^x(t-)$, is resampled at
$$ X_k^x(t) := (1-\Theta_k) X^x_{k-1}(t-)+\Theta_k X^x_{k+1}(t-)\;,$$
where $\Theta_k$ is an independent random variable on $[0,1]$ drawn according to some fixed distribution, possibly depending on $k$. This resampling mechanism clearly preserves the ordering of the particles.\\
It is convenient to view each configuration $x\in S_N$ as a height function $k \mapsto x_k$, which is a non-decreasing map from $\{0,\ldots,N\}$ to $[0,N]$ that is bound to $(0,0)$ and $(N,N)$, we refer to Figure \ref{Fig_expect} for an illustration.\\

From now on, we assume that the r.v.~$\Theta_k$ follow Beta $(\alpha_k,\alpha_{k+1})$ laws for some sequence of parameters $(\alpha_k; k=1,\ldots,N)$. Actually, this is the only choice of resampling laws for which the process is reversible w.r.t.~an invariant measure that has a product structure (we refer to Section \ref{Sec:Ppties} for more details).\\

Our goal is to investigate the asymptotic in $N$ of the mixing times of this model. Let $P_t^x$ denote the law of $X^x(t)$ starting from the configuration $x\in S_N$ and let $\pi_N$ be the invariant measure. The $\eps$-mixing time is defined as
$$ t^N_\text{mix}(\eps)=\inf\left\{t\geqslant 0, \sup_{x\in S_N} \Vert P^x_t-\pi_N\Vert_{\text{TV}}\leqslant \eps\right\}\;.$$
In words, this is the time needed for the total variation distance of the process to equilibrium, starting from the ``worst'' initial condition, to pass below some given threshold $\eps\in(0,1)$.

\subsection{Existing results}

This question has already been addressed for \emph{symmetric} instances of the model. In 2005, Randall and Winkler~\cite{RW05} considered the particular case where $\alpha_k = 1$ for all $k\in \{1,\ldots,N\}$: the beta laws are then merely uniform laws so that each coordinate $X^x_k$ is resampled at rate $1$ uniformly over the segment $[X^x_{k-1},X^x_{k+1}]$ formed by its neighbours, and the invariant measure is the Lebesgue measure on $S_N$. In this setting, they established a concentration phenomenon for the mixing times, often referred to as \emph{pre-cutoff}: they showed the existence of two constants $0 < C < C'$ such that for all $\eps \in (0,1)$ and for all $N\ge 1$ large enough
$$ CN^2\log(N)\leqslant t^N_\text{mix}(\eps)\leqslant C'N^2\log(N)\;.$$
Recently, Caputo, Labb\'e and Lacoin~\cite{CLL} sharpened this result by establishing a \emph{cutoff phenomenon}: for all $\eps\in (0,1)$, as $N\to\infty$
$$ t^N_{\text{mix}}(\eps) \sim \frac{N^2 \log N}{\pi^2}\;.$$
Actually, the result in~\cite{CLL} is proved in the more general setting where $\alpha_k = \alpha \ge 1$ for all $k\in \{1,\ldots,N\}$: that is to say, in the situation where the resampling laws are \emph{symmetric} (the two parameters of the beta distributions are equal) and \emph{unimodal} (this parameter is larger than or equal to $1$). The invariant measure is still explicit in this setting and its large scale behaviour is quite elementary: the $k$-th coordinate at equilibrium is centered at $k$ with fluctuations of order $\sqrt{N}$ with gaussian laws in the limit. Note that the value $\alpha$ does not affect the order of the mixing times nor the prefactor.\\


\subsection{Main results}

In the present work, we investigate the situation where the resampling laws are \emph{asymmetric}, that is, the situation in which the means of the beta laws are different from $1/2$. We restrict ourselves to the case where the means of the beta laws are all the same and adopt the parametrisation
$$ \alpha_k = \alpha_1 \Big(\frac{1+\lambda}{1-\lambda}\Big)^{k-1}\;,\quad\text{with }\lambda=\lambda(N) > 0 \text{ and } \alpha_1 \ge 1\;.$$
The resampling laws are then asymmetric, unimodal beta laws. The dynamics tends to resample the particles closer to their left neighbours. The invariant measure remains explicit, and the behaviour of the height function is quite different from that in the symmetric case, see Figure \ref{Fig_expect}.

\begin{figure}
	\centering
	\begin{minipage}[b]{0.3\linewidth}
	\includegraphics[width = 4cm]{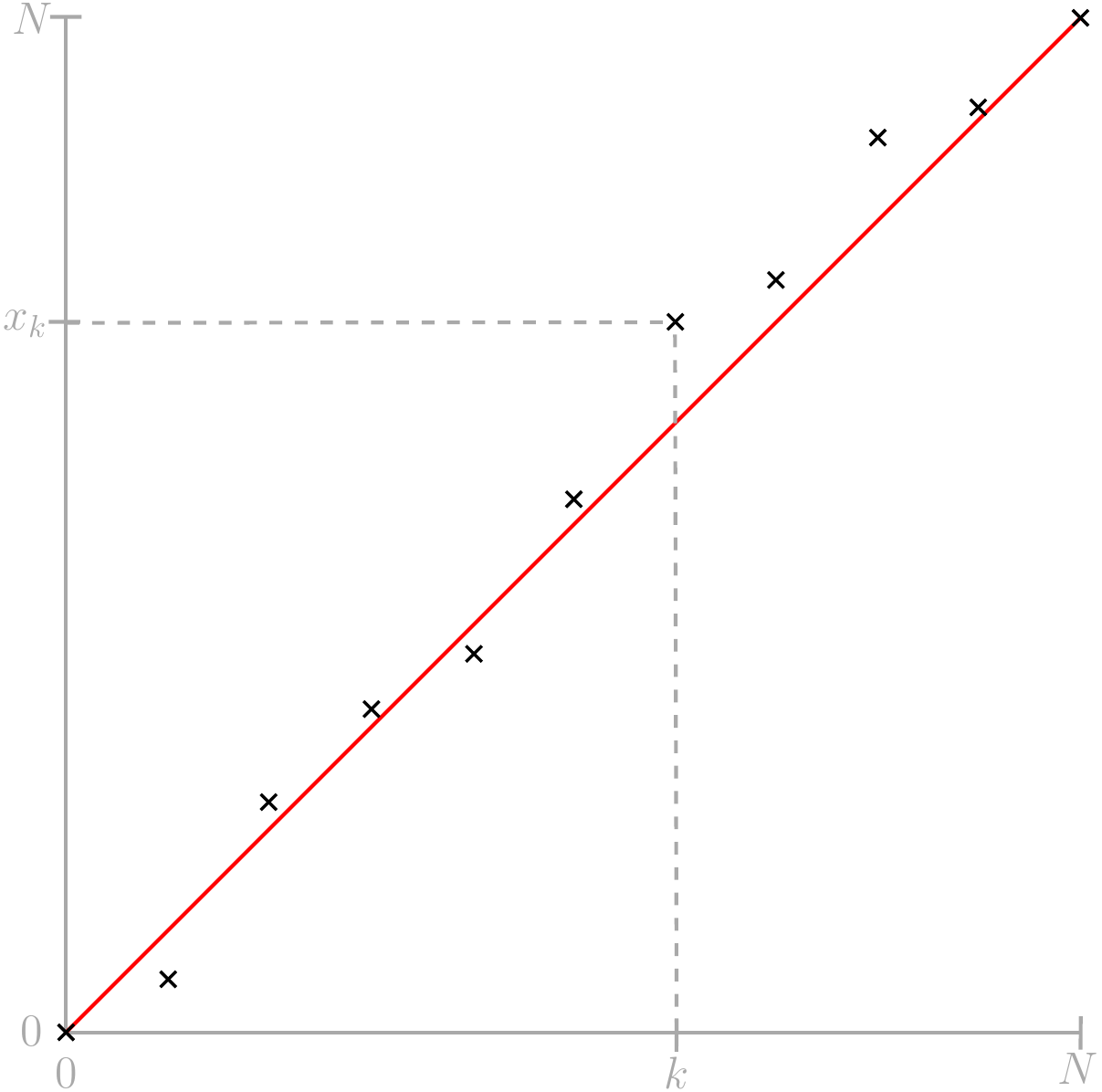}
	\end{minipage}\hspace{1.4cm}
	\begin{minipage}[b]{0.3\linewidth}
	\includegraphics[width = 4cm]{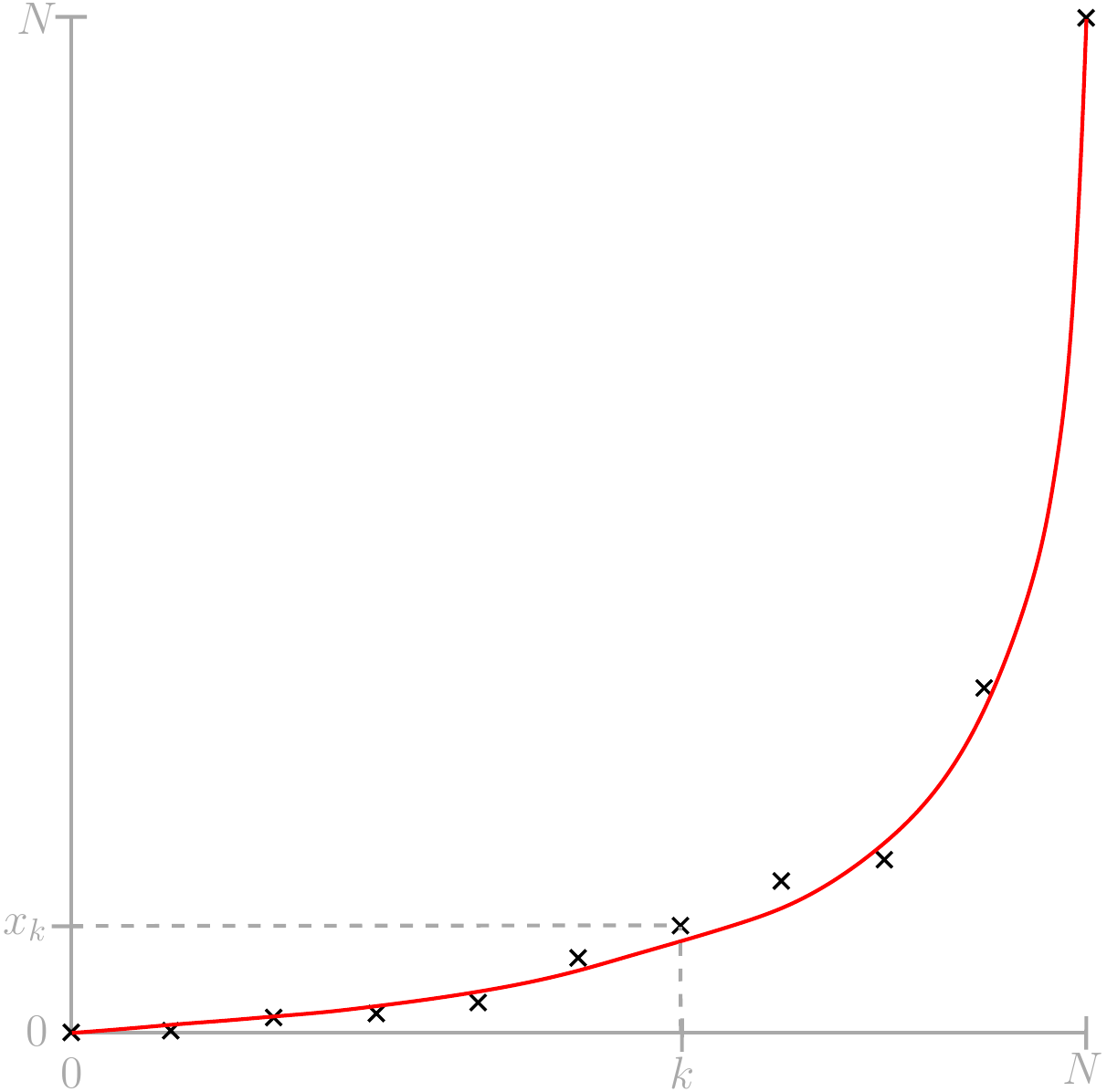}
	\end{minipage}
\caption{In red, the profile of the expectation under the invariant measure, and in black a typical height function. On the left, the symmetric case ($\lambda = 0$) and on the right, the asymmetric case ($\lambda = O(1)$).}\label{Fig_expect}
\end{figure}

It happens that the pattern of convergence to equilibrium depends also drastically on the strength (w.r.t.~$N$) of the asymmetry. The main results of the present article are as follows.

\begin{thm}\label{Th:2}
	Assume that the asymmetry parameter $\lambda > 0$ is independent of $N$. Then for all $\eps\in (0,1)$ and all $\delta > 0$, provided that $N$ is large enough
	$$ (1-\delta)\frac{N}{\lambda}  \leqslant t^N_{\text{mix}}(\eps) \leqslant (1+\delta) \frac{N\log(\frac{1+\lambda}{1-\lambda})}{1-\sqrt{1-\lambda^2}}\;.$$
\end{thm}

\noindent This establishes a pre-cutoff phenomenon in the situation where the asymmetry does not depend on $N$. In the situation where the asymmetry parameter vanishes but not too fast, we establish a cutoff phenomenon:

\begin{thm}\label{Th:Main}
	Assume that the asymmetry parameter satisfies $\lambda \to 0$ and $\lambda \gg \log N / N$ as $N\to\infty$. Then for all $\eps\in (0,1)$
	$$t^N_{\text{mix}}(\eps)\sim \frac{4N}{\lambda}\;,\quad N\to\infty\;.$$
\end{thm}

The regime where the asymmetry vanishes faster will be covered in a companion paper by the second named author, we quote the result here for completeness.

\begin{thm}[\cite{Enguerand}]
	Assume that the asymmetry parameter satisfies $\lambda \ll \log N / N$ as $N\to\infty$. Then for all $\eps\in (0,1)$
	$$t^N_{\text{mix}}(\eps)\sim \frac{\log N}{2 \gap_N}\;,\quad N\to\infty\;,$$
	and the spectral gap of the generator of the dynamics, denoted $\gap_N$, is given for every $N\ge 2$ by
	$$ \gap_N = 1 - \sqrt{1-\lambda^2} \cos(\frac{\pi}{N})\;.$$
\end{thm}

\begin{rem}
	In the regime where $\lambda\sim a\frac{\log(N)}{N}$ for some parameter $a>0$, we conjecture that a cutoff phenomenon holds and that the mixing times are equivalent to $\frac{\log(N)}{2\gap_N}+\frac{4N}{\lambda}$. Actually, one can rigorously prove that the mixing times are at most given by this expression.
\end{rem}

The proof of Theorems \ref{Th:2} and \ref{Th:Main} consists in establishing separately matching upper and lower bounds on the mixing times.

\subsubsection*{The upper bound.}
To prove the upper bound, we consider a monotone coupling of $X^x$ and $X^y$, that is, a coupling that preserves the ordering of the height functions provided the original configurations $x$ and $y$ were ordered. We follow a two-steps strategy, originally introduced by Randall and Winkler~\cite{RW05} in the symmetric case: (1) by a spectral argument, one shows that the area between the two height functions is ``small'' by the putative mixing time, (2) one shows that, if the area is small, then within a small time with large probability the area hits $0$. We apply it twice: first to $x= \max$ and $y$ drawn according to the invariant measure, and second to $y=\max$ and $x$ arbitrary. Here $\max$ denotes the maximal configuration
$$ {\max}_k = N\;,\quad \forall k\in \{1,\ldots,N-1\}\;.$$
By the triangle inequality, this is sufficient to obtain an upper bound on the total variation distance between $P^x_t$ and $P^\pi_t$ for any $x$, and thus on the mixing times.\\
Let us mention that in the asymmetric setting of Theorem \ref{Th:Main}, this strategy provides not only the right order for the mixing times but also the precise prefactor. This is to be compared with the symmetric situation where this strategy allowed Randall and Winkler~\cite{RW05} to obtain the right order but not the precise prefactor: Caputo, Labb\'e and Lacoin~\cite{CLL} then refined this strategy to obtain the precise prefactor - this is unnecessary in the present setting.

\subsubsection*{The lower bound.}
A natural strategy to establish a lower bound would be to apply Wilson's method: taking $f:S_N\to\R$ to be the eigenfunction associated to the spectral gap of the dynamics (which is explicit in our model), one compares the expectation and standard-deviation of $f(X^x(t))$ with the same statistics under the invariant measure. If the difference of expectations is much larger than the sum of standard-deviations, then necessarily the process $X^x(t)$ is far from equilibrium so that $t$ is a lower bound on the mixing times. A natural choice for $x$ is the maximal height function $\max$. It turns out that in the symmetric setting, this method~\cite{RW05, CLL} allows to obtain an optimal lower bound. Unfortunately, in the asymmetric setting considered in Theorem \ref{Th:Main}, this method does not provide an accurate lower bound simply because the standard-deviation of $f(X^x(t))$ is quite hard to estimate.

We follow another approach: we establish a hydrodynamic limit for our model and show that at a time smaller than the putative mixing time, the hydrodynamic limit is far from its equilibrium. This strategy was employed for the (weakly) asymmetric simple exclusion process by Labb\'e and Lacoin~\cite{LLASEP,LLWASEP} and is intimately related to the asymmetric nature of the model. Indeed, in more symmetric situations, one can establish a hydrodynamic limit (typically, a heat equation) but it generally evolves at a slower time scale than the actual mixing times and is therefore not useful to get accurate lower bounds.

\subsubsection*{The hydrodynamic limit(s).}
Actually, we establish two distinct hydrodynamic limits. The first, natural one arises as the limit of
$$g_N(x,t) := \frac1{N} X_{\lfloor xN\rfloor}^{\max}(N t / \lambda)\;,\quad x\in [0,1]\;, t\ge 0\;.$$
Simple computations at the level of the generator suggest that it should be given by the (viscosity) solution of the equation
$$ \begin{cases} \partial_t g + \partial_x g = 0\;,\quad x\in(0,1)\;,t>0\\
	g(t,0) = 0\;,\quad g(t,1) = 1\;,\quad g(0,x) = 0\;.
	\end{cases}$$
which happens to be given by $g(x,t) = \mathbf{1}_{[t,1]}(x)$. This is the content of the following result.

\begin{prop}\label{Prop:HydroNaive}
Assume $\lambda \gg 1/N$. Then for all $x \in (0,1)$ and $t > 0$ such that $x\ne t$, $g_N(x,t)$ converges in probability to $\mathbf{1}_{[t,1]}(x)$ as $N\to\infty$.
\end{prop}

Simple concentration estimates on the invariant measure combined with this convergence thus yield a lower bound on the mixing times given by $N/\lambda$. This is enough for Theorem \ref{Th:2}, which covers the case where $\lambda$ is independent of $N$, but for Theorem \ref{Th:Main} it falls short of the expected result by a factor $4$.\\

This can be explained easily. The equilibrium profile is ``degenerate'': under $\pi_N$, the $k$-th coordinate $x_k$ is of order $N r^{k-N}$, where $r = (1+\lambda)/(1-\lambda)$, and this quantity vanishes except for $k$ close enough to $N$. The convergence of $g_N$ to $g$ only implies that after time $N/\lambda$, the $k$-th coordinate is negligible compared to $N$ but it does provide any further control on how close it is from its equilibrium value. In other words, the macroscopic scaling involved in the definition of $g_N$ is too rough to capture the details of the microscopic equilibrium.\\

This discussion motivates the introduction of a transformation of the height function that produces non-trivial quantities in the large scale limit. This is achieved by the following non-linear transformation:
\begin{align*}
	T:[0,N]&\too [0,1]\\
	u&\longmapsto \frac{-1}{N}\log_r\left(\frac{u}{N} (1-r^{-N})+r^{-N}\right)\;.
\end{align*}
Note that $k\mapsto \pi(x_k) \sim N r^{k-N}$, which is degenerate in the limit $N\to\infty$, is mapped onto $k\mapsto T(\pi(x_k)) = 1 - k/N$. We then set for all $x\in [0,1]$ and all $t\ge 0$
$$h_N(x,t) := T\left(X^{\max}_{\lfloor xN\rfloor}\Big(t\frac{N}{\lambda}\Big)\right)\;,\quad x\in [0,1]\;,t\geqslant 0\;.$$

\begin{thm}\label{Th:PDE}
	Assume that $\log N / N \ll \lambda \ll 1$. For all $\eps>0$ and $t\geqslant 0$, the following convergence holds in probability
	$$\lim_{N\to\infty} \sup_{x\in [\eps,1]} |h_N(x,t) - S(x,t)| = 0\;,$$
	where $S$ is the unique viscosity solution of
	$$ \begin{cases} \partial_t f+\partial_x f+(\partial_x f)^2=0\;, \quad x\in (0,1)\;,\quad t>0\;,\\
		f(t,0) = 1\;,\quad f(t,1) = 0\;,\quad f(0,x) = 0\;,\end{cases}$$
	which happens to be explicitly given by
	$$ S(x,t)=\min\left(1-x,\frac{1}{4t}((t-x)_+)^2\right) \;.$$
\end{thm}

This convergence provides the desired lower bound on the mixing times. Indeed, $S$ reaches the invariant profile $x\mapsto 1-x$ at time $4$ (that is, $4N/\lambda$ in the original time scale) and it is a simple task to check that any macroscopic profile which is distinct from this invariant profile corresponds to microscopic configurations that are far from equilibrium.\\

The PDE that appears in the statement is a non-linear Hamilton-Jacobi equation. The notion of viscosity solutions that we employ here is originally due to Crandall and Lions~\cite{CrandallLions}. However, the boundary conditions (which naturally arise from the initial profile, and the fact that $X^{\max}_0(t) = 0$ and $X^{\max}_N(t) = N$) are discontinuous at $(0,0)$, so that, to the best of our knowledge this PDE does not fall into the scope of general existence/uniqueness results. In Section \ref{Sec:HJ} we prove uniqueness of the viscosity solution, and show that $S$ is a solution: this section relies on PDE arguments only and is self-contained.\\

The proof of the convergence of $h_N$ towards $S$ is our main technical achievement. There exist already several works dedicated to proving convergence of discrete stochastic models to Hamilton-Jacobi equations, starting with the seminal articles of Sepp\"al\"ainen~\cite{Seppa} and Rezakhanlou~\cite{Reza2}, see also~\cite{Reza1,Toninelli,Zhang} for instance. In these references, after having showed tightness, the authors identify the limit when the process starts from a product stationary measure and then extend this identification for any initial conditions by approximation. These techniques do not seem easy to adapt to our setting: first, we work in finite volume so dealing with the boundary conditions would require substantial modifications; second, our dynamic is much less regular than those considered in these works, in particular the size of the jumps is not bounded but can be of order $N$ thus making the control of the space-time increments of the process delicate. Therefore we follow another approach, which does not seem to have been used for proving convergence of stochastic models to Hamilton-Jacobi equations: we identify discrete-space approximation of the PDE at the level of our model, and we show that the sequence of solutions converge to a viscosity solution of our PDE. It has some flavour coming from numerical schemes, and our proof exploits the stability of the theory of viscosity solutions under limits.\\

Let us now give some more details. A direct approach would consist in showing that the stochastic evolution equations of $T(X^{\max})$ are asymptotically close to the above PDE, but the non-linearity of $T$ combined with the non-trivial resampling mechanism make the corresponding computations untractable. However, it turns out that $T(\E[X^{\max}])$ satisfies a space-discretization of the PDE of the statement and we prove in Subsection \ref{Subsec:CVODE} that it converges to a viscosity solution. A simple computation then shows that $T(X^{\max}) \ge T(\E[X^{\max}])$ with large probability, so it remains to control $T(X^{\max})$ from above. To that end, we couple $X^{\max}$ with a Markov process $M$ that remains below $X^{\max}$ (equivalently, $T(M)$ remains above $T(X^{\max})$) with large probability, but whose resampling variables are deterministic. We then prove that $\E[T(M)]$ also converges to a viscosity solution of the PDE. The uniqueness of the viscosity solution allows to conclude.

\bigskip

The remaining of this article is organized as follows. In Section \ref{Sec:Ppties} we introduce in more details the model and collect some estimates for the sequel. In Section \ref{Sec:Upper} we prove the upper bounds on the mixing times. In Section \ref{Sec:HJ} we provide the solution theory of the above PDE. Finally in Section \ref{Sec:Lower} we prove Theorem \ref{Th:PDE} and Proposition \ref{Prop:HydroNaive} and thus conclude the lower bounds on the mixing times.

\subsection*{Acknowledgements}
The work of C.L.~is supported by the project SINGULAR ANR-16-CE40-0020-01.

\section{Some properties of the model}\label{Sec:Ppties}

This section introduces precisely the model and collects several properties and estimates that will be needed later on. We start by considering the dynamics on a larger state-space in order to identify a natural class of resampling measures.

\subsection{Dynamics on the unconstrained simplex and reversibility}

We first introduce our dynamics on a larger state-space where the value at the endpoint is not specified
$$S_N^+ := \big\{x=(x_0,...,x_N)\in\R^{N+1}\;,\quad 0=x_0\leqslant x_1\leqslant...\leqslant x_N \big\}\;.$$
Given some continuous\footnote{We exclude the presence of atoms in the resampling measures in order to avoid deterministic resampling events.} probability measures $(\rho_k)_{k=1,\ldots,N-1}$ on $[0,1]$, referred to as the \emph{resampling measures} in the sequel, we consider the continuous-time Markov process with generator
$$ \Lf^+ = \sum_{k=1}^{N-1} (T_k - id)\;,$$
where
$$ T_k f(x) := \int_{[0,1]} f(x^{(k,u)}) \rho_k(du)\;,$$
and 
$$x^{(k,u)}_i=\left\{\begin{array}{ll}
x_i&\text{if }i\not=k\\
x_{i-1}+u(x_{i+1}-x_{i-1})&\text{if }i=k
\end{array} \right. .$$

It is sometimes convenient to deal with the following alternative representation of the state space: for any $x\in S_N^+$, let $\eta \in \R^N_+$ be defined through
$$ \eta_k := x_k - x_{k-1}\;,\quad k=1\;,\ldots, N\;.$$
This mapping defines a bijection between $S_N^+$ and $\R^N_+$. We then keep the same notation for our generator when acting on functions on $\R^N_+$. In particular
$$ T_k f(\eta) := \int_{[0,1]} f(\eta^{(k,u)}) \rho_k(du)\;,$$
and 
$$\eta^{(k,u)}_i=\left\{\begin{array}{ll}
\eta_i&\text{if }i\notin \{k,k+1\}\\
u(\eta_i+\eta_{i+1})&\text{if }i=k\\
(1-u)(\eta_i+\eta_{i+1})&\text{if }i=k+1
\end{array} \right. .$$

In probabilistic terms, the process $(X_t^x,t\ge 0)$, or equivalently $(\eta_t^\eta,t\ge 0)$, with generator $L^+$ evolves as follows: for any $k$, at rate $1$ one draws a r.v.~$U$ according to $\rho_k$ and one resamples $X_k(t-)$ on $X_{k-1}(t-) + U(X_{k+1}(t-) - X_{k-1}(t-))$, or equivalently $(\eta_k(t-),\eta_{k+1}(t-))$ on $(U(\eta_k(t-)+\eta_{k+1}(t-)), (1-U)(\eta_k(t-)+\eta_{k+1}(t-)))$.\\

The following result identifies the family of resampling measures for which the dynamics preserves product laws. Actually, the result is spelled out at the level of a single resampling operator $T_k$.

\begin{prop}
Fix $k\in \{1,\ldots,N-1\}$. If $\mu_k \otimes \mu_{k+1}$ is invariant for the operator $T_k$ (viewed as acting on the pair $(\eta_k,\eta_{k+1})$), that is,
$$ (\mu_k \otimes \mu_{k+1}) T_k = \mu_k \otimes \mu_{k+1}\;,$$
then there exist $\alpha_k, \alpha_{k+1},s \in (0,\infty)$ such that $\rho_k$ is the $\beta(\alpha_k,\alpha_{k+1})$ distribution, $\mu_k$ the $\Gamma(\alpha_k,s)$ distribution and $\mu_{k+1}$ the $\Gamma(\alpha_{k+1},s)$ distribution.
\end{prop}
\begin{proof}
Let $(\eta_k,\eta_{k+1})$ be a random variable with distribution $\mu_k \otimes \mu_{k+1}$. Let $\Theta$ be a random variable independent of $(\eta_k,\eta_{k+1})$ with distribution $\rho_k$. By invariance of $\mu_k \otimes \mu_{k+1}$ for the operator $T_k$, we deduce the equality in distribution:
$$(\eta_k,\eta_{k+1})=\big(\Theta(\eta_k+\eta_{k+1}),(1-\Theta)(\eta_k+\eta_{k+1})\big)\;,$$
which in turn implies
$$\left( \frac{\eta_k}{\eta_k+\eta_{k+1}},\eta_k+\eta_{k+1}\right)=(\Theta,\eta_k+\eta_{k+1}).$$
As $\Theta$ is independent of $(\eta_k,\eta_{k+1})$, we deduce that $\frac{\eta_k}{\eta_k+\eta_{k+1}}$ and $\eta_k+\eta_{k+1}$ are independent. We conclude with the following property of the Gamma laws~\cite{lukacs1955}: if $X,Y$ are two independent random variables and if $\frac{X}{X+Y}$ and $X+Y$ are independent, then there exist $\alpha,\alpha',s > 0$ such that 
\begin{align*}
X\sim\Gamma(\alpha,s),\quad
Y\sim\Gamma(\alpha',s),\quad
\frac{X}{X+Y}\sim\beta(\alpha,\alpha')\;.
\end{align*}
\end{proof}

This being given, suppose that $\mu=\otimes_{k=1}^{N}\mu_k$ is a reversible measure for the above dynamics on $\R_N^+$. It can be checked that necessarily $\mu_k \otimes \mu_{k+1}$ is invariant for the operator $T_k$, and therefore we deduce from the proposition that there exist $\alpha_1,\ldots,\alpha_N,s \in (0,\infty)$ such that $\rho_{k}$ is the $\beta(\alpha_k,\alpha_{k+1})$ distribution and $\mu_k$ is the $\Gamma(\alpha_k,s)$ distribution. Conversely if $\mu$ is a product of $\Gamma(\alpha_k,s)$ distributions and if each $\rho_{k}$ is a $\beta(\alpha_k,\alpha_{k+1})$ distribution then it is not hard to check that each operator $T_k$ (and therefore the generator $\Lf^+$) is self-adjoint in $L^2(\mu)$.\\
This shows that the dynamics on $\R^N_+$ is reversible w.r.t.~a product measure if and only if the $\rho_k$ are $\beta(\alpha_k,\alpha_{k+1})$ distributions for some sequence $(\alpha_k)_k$, and the invariant measure then consists of a product of $\Gamma(\alpha_k,s)$ distributions.\\

Observe that the last particle $x_N$ (or equivalently, $\eta_1+\ldots+\eta_N$) is invariant under the dynamics. So the set 
$S_N^a=\{x\in \R^{N+1},0=x_0\leqslant x_1\leqslant...\leqslant x_N=a\}$ is left invariant by the dynamics. From now on, we will focus on $S_N=S_N^N$ although the choice $a=N$ is arbitrary.

\subsection{The model on the simplex}

From now on, we consider two parameters $\alpha_1 \geqslant 1$ and $\lambda \in [0,1)$ and we set
$$ \alpha_k := \alpha_1 r^{k-1}\;,\quad k\in \{1,\ldots,N-1\}\;,\quad r := \frac{1+\lambda}{1-\lambda}\;.$$
For every $k\in \{1,\ldots,N-1\}$ we let $\rho_k$ be the $\beta(\alpha_k,\alpha_{k+1})$ distribution. In this framework, the resampling laws all have mean $(1-\lambda)/2$. Note that we allow $\lambda$ to depend on $N$, but not $\alpha_1$. We then consider the Markov process generated by the restriction $\Lf$ of $\Lf^+$ to the set of functions from $S_N$ to $\R$. We denote by $(X^x(t),t\ge 0)$ this process when starting from $x\in S_N$.

\begin{lem}\label{Lem:InvMeas}
The dynamics on $S_N$ is reversible w.r.t.~the following probability measure:
$$\pi_N(\dx x)=\frac{\Gamma(\alpha)}{N^{\alpha-1}}\Big(\prod_{i=1}^N \frac{(x_i-x_{i-1})^{\alpha_i-1}}{\Gamma(\alpha_i)} \mathbf{1}_{x_{i-1} \leqslant x_i}\Big) \dx x_1\cdots \dx x_{N-1} \;,$$
where $\alpha := \alpha_1+\ldots+\alpha_N$.\\
Under $\pi_N$, $\frac{\eta_{k}+\eta_{k+1}}{N} \sim \beta(u,\alpha-u)$ with $u=\alpha_k+\alpha_{k+1}$ and $\frac{x_k}{N} \sim \beta(u,\alpha-u)$ with $u=\sum_{i=1}^k \alpha_i$.
\end{lem}

\begin{rem}
	The fact that this is the only invariant measure of the dynamics will be a consequence of the proof of the upper bound on the mixing times.
\end{rem}

\begin{proof}
Fix some $s > 0$, let $\mu_k$ be the $\Gamma(\alpha_k,s)$ distribution and set $\mu=\otimes_{k=1}^{N}\mu_k$. We already know that $\mu$ is reversible for the dynamics on $S_N^+$. Note that under $\mu$, $x_N$ is a $\Gamma(\alpha,s)$ r.v. Now let $\pi_N$ be the measure $\mu$ conditioned to $x_N = N$, or equivalently to $\eta_1+\ldots+\eta_N = N$. If we let $\phi(x,b,s)$ be the density at $x$ of the $\Gamma(b,s)$ law, then $\pi_N$ admits a density (in the $\eta$ variables) on $S_N$ given by
$$ \frac{\prod_{i=1}^N \phi(\eta_i,\alpha_i,s)}{\phi(N,\alpha,s)}\;.$$
Since $\eta_1+\ldots+\eta_N = N$ on $S_N$, a simple computation yields the asserted density. To check that our dynamics on $S_N$ is reversible w.r.t.~$\pi_N$, it suffices to check that each $T_k$ is self-adjoint in $L^2(\pi_N)$. This easily follows from the fact that $T_k$ is the orthogonal projection on the set of functions in $L^2(\pi_N)$ that only depend on $\{x_i: i\ne k\}$.\\
We turn to the second part of the statement. Under $\mu$, the r.v.~$(\eta_k + \eta_{k+1})/ x_N$ and $x_N$ are independent with respective distributions $\beta(\alpha_k+\alpha_{k+1},\alpha - (\alpha_k+\alpha_{k+1}))$ and $\Gamma(\alpha,s)$. Since $\pi_N$ is obtained by conditioning $\mu$ to $\{ x_N = N\}$ we deduce that under $\pi_N$ the r.v.~$(\eta_k+\eta_{k+1})/ x_N$ is distributed according to a $\beta(\alpha_k+\alpha_{k+1},\alpha - (\alpha_k+\alpha_{k+1}))$ law, and this yields the desired result since $(\eta_k+\eta_{k+1})/ x_N = (\eta_k+\eta_{k+1})/ N$ under $\pi_N$. A similar argument yields the law of $x_k/N$ under $\pi_N$.
\end{proof}

\begin{rem}
	The model studied in the present article can be seen as a finite volume version of the Random Average Process studied in~\cite{Balazs}. The hydrodynamic limit obtained in Proposition \ref{Prop:HydroNaive} is in line with the hydrodynamic limit obtained in that article (but in infinite volume).
\end{rem}

\subsection{Some properties of the beta law}

We let $U_k$ be a r.v.~with a $\beta(\alpha_k,\alpha_{k+1})$ distribution and $m_k$ be the maximum of the associated density function. Recall that $\lambda$, and therefore $r$, is allowed to depend on $N$ consequently we will control the dependence in $\lambda$ of the quantities of interest. On the other hand, the value of $\alpha_1$ is fixed and we do not control the dependence in this parameter.

\begin{lem}\label{esperance loi beta}
There are two constants $0 < C_1 < C_2$ such that for all $\lambda$ in a compact set of $[0,1)$, all $k\in \{0,\ldots,N-1\}$ and all $N\ge 2$, the following hold:
\begin{align*}
\E(U_k)=\frac{\alpha_k}{\alpha_k+\alpha_{k+1}}=\frac{1}{1+r} = \frac{1-\lambda}{2}\;,\\
C_1 r^{-k}\leqslant \Var(U_k)\leqslant C_2 r^{-k}\;,\\
C_1 r^{k/2}\leqslant m_k\leqslant C_2 r^{k/2}\;.
\end{align*}
\end{lem}
\begin{proof}
The expression of the expectation is a standard result on beta laws. Regarding the variance, we have
\begin{align*}
\Var(U_k)&=\frac{\alpha_k\alpha_{k+1}}{(\alpha_k+\alpha_{k+1})^2(\alpha_k+\alpha_{k+1}+1)}\\
&=\frac{r}{(1+r)^2(\alpha_1 r^{k-1}+\alpha_1 r^{k}+1)}\;,
\end{align*}
and the upper and lower bounds on the variance easily follow. We turn to the maximum of the density, which is achieved at $x_k := \frac{\alpha_k-1}{\alpha_k+\alpha_{k+1}-2}$. We find
\begin{align*}
m_k&=\frac{\Gamma(\alpha_k+\alpha_{k+1})}{\Gamma(\alpha_k)\Gamma(\alpha_{k+1})}\left(\frac{\alpha_k-1}{\alpha_k+\alpha_{k+1}-2}\right)^{\alpha_k-1}\left(\frac{\alpha_{k+1}-1}{\alpha_k+\alpha_{k+1}-2}\right)^{\alpha_{k+1}-1}\;.
\end{align*}
Set $g(x) := \Gamma(x) / (\sqrt{2\pi}x^{x-1/2}e^{-x})$. A simple computation then shows that
\begin{align*}
	\frac{m_k}{r^{k/2}}&= \frac{(\alpha_k + \alpha_{k+1})^{3/2}}{r^{k/2} \sqrt{2\pi \alpha_k \alpha_{k+1}}} \frac{g(\alpha_k + \alpha_{k+1})}{g(\alpha_k) g(\alpha_{k+1})}  \frac{(1-\frac{1}{\alpha_k})^{\alpha_k-1}(1-\frac{1}{\alpha_{k+1}})^{\alpha_{k+1}-1}}{(1-\frac{2}{\alpha_k+\alpha_{k+1}})^{\alpha_k+\alpha_{k+1}-2}}\\
	&= \frac{\sqrt{\alpha_0}(1+r)^{3/2}}{\sqrt{2\pi r}} \frac{g(\alpha_k + \alpha_{k+1})}{g(\alpha_k) g(\alpha_{k+1})}  \frac{(1-\frac{1}{\alpha_k})^{\alpha_k-1}(1-\frac{1}{\alpha_{k+1}})^{\alpha_{k+1}-1}}{(1-\frac{2}{\alpha_k+\alpha_{k+1}})^{\alpha_k+\alpha_{k+1}-2}}\;.
\end{align*}
Since $g$ is bounded on $[1,\infty)$ and converges to $1$ at $+\infty$, and since $\alpha_k = \alpha_1 r^{k-1}$ is always larger than or equal to $1$, the asserted bound follows.
\end{proof}

The next lemma relates the distribution functions of the gamma and beta distributions.
\begin{lem}\label{absolue continuite}
We let $U\sim\beta(u,v)$ and $Z\sim\Gamma(u,u+v)$. For any $\varrho\in[0,1)$, there exists a constant $C_\varrho > 0$ such that for all $u,v\geqslant 1$ such that $u/(u+v) \leqslant\varrho$ and all $t\in[0,1]$ :
$$\P(U\leqslant t)\leqslant C_\varrho\,\P(Z\leqslant t)\;.$$
\end{lem}

\begin{proof}
Note that the map $x\mapsto (1-x)^{v-1}e^{(u+v)x}$ reaches its maximum over $[0,1]$ at $x=\frac{u+1}{u+v}$. We compute:
\begin{align*}
\P(U\leqslant t)&=\frac{\Gamma(u+v)}{\Gamma(u)\Gamma(v)}\int_0^tx^{u-1}(1-x)^{v-1}\dx x\\
&=\frac{\Gamma(u+v)}{\Gamma(u)\Gamma(v)}\int_0^tx^{u-1}e^{-(u+v)x}(1-x)^{v-1}e^{(u+v)x}\dx x\\
&\leqslant\frac{\Gamma(u+v)}{\Gamma(u)\Gamma(v)}\left(\frac{v-1}{u+v}\right)^{v-1}e^{u+1}\int_0^tx^{u-1}e^{-(u+v)x}\dx x\\
&\leqslant\frac{\Gamma(u+v)}{(u+v)^u\Gamma(v)}\left(\frac{v-1}{u+v}\right)^{v-1}e^{u+1}\,\P(Z\leqslant t)\;.
\end{align*}
To conclude, it suffices to show that
$$c(u,v) := \frac{\Gamma(u+v)}{(u+v)^u\Gamma(v)}\left(\frac{v-1}{u+v}\right)^{v-1}e^{u+1}$$
is bounded uniformly over all $u\geqslant 1,v\geqslant 1$ such that $\frac{u}{u+v}\leqslant\varrho$. By continuity, this obviously holds over compact sets in $u$ and $v$. It remains to bound $c(u,v)$ at infinity. Note that the condition on $u$ and $v$ implies that if $u+v\to\infty$ then $v\to\infty$: consequently, it is sufficient to bound $c(u,v)$ when $v\to\infty$ and uniformly over all $u \ge 1$.\\
Recall that $\Gamma(x)\sim \sqrt{2\pi}x^{x-1/2}e^{-x}$ as $x$ goes to $\infty$. As $v\to\infty$ and uniformly over all $u \ge 1$, we have
\begin{align*}
c(u,v)&\sim\frac{(u+v)^{u+v-1/2}e^{-u-v}}{(u+v)^uv^{v-1/2}e^{-v}}\left(\frac{v-1}{u+v}\right)^{v-1}e^{u+1}\\
&\sim e(u+v)^{1/2}\frac{(v-1)^{v-1}}{v^{v-1/2}}\\
&\sim e(u+v)^{1/2}v^{-1/2}\;.
\end{align*}
As $e(u+v)^{1/2}v^{-1/2}\leqslant e(1-\varrho)^{-1/2}$ we deduce that $c(u,v)$ is uniformly bounded.
\end{proof}

We now state a deviation estimate on the beta law. Recall from Lemma \ref{esperance loi beta} that the variance of $U_k$ is of order $r^{-k}$.

\begin{lem}\label{concetration beta}
For any $c>0$, there exists $C > 0$ such that for all $\lambda$ in a compact set of $[0,1)$, all $k\in\{0,\ldots,N-1\}$ and all $N\ge 2$
$$\P\left(U_k \leqslant \E(U_k)-C\log(N) r^{-k/2}\right)\leqslant N^{-c}\;.$$
\end{lem}

\begin{proof}
Set $u_k := \E(U_k)-C\log(N)r^{-k/2}$ for some constant $C>0$ that will be adjusted later on. Since $\frac{\alpha_k}{\alpha_k+\alpha_{k+1}}=\frac{1}{1+r}$ and $r\ge 1$, we can apply Lemma \ref{absolue continuite} with $\varrho=\frac{1}{2}$ and $Z_k \sim \Gamma(\alpha_k,\alpha_k+\alpha_{k+1})$ to obtain
$$\P\left( U_k\leqslant u_k\right) \leqslant C_\varrho \,\P\left(Z_k \leqslant u_k \right)\;.$$
The exponential moments of the gamma distribution are given by $\E[e^{tZ_k}] = \big( 1 - \frac{t}{\alpha_k+\alpha_{k+1}}\big)^{-\alpha_k}$. By the Chernoff inequality we deduce:
$$\P\left( Z_k \leqslant u_k\right)\leqslant \exp\left(-I(u_k)\right)\;,$$
with
$$ I(x)=\sup_{t\leqslant 0}\left\{ xt + \alpha_k \log\left(1-\frac{t}{\alpha_k+\alpha_{k+1}}\right) \right\}\;.$$
Taking $t = -r^{k/2}C\log(N)$, we find
\begin{align*}
I(u_k)&\geqslant -\E(U_k)r^{k/2}C\log(N)+(C\log(N))^2+\alpha_k\log\left(1+\frac{r^{k/2}C\log(N)}{\alpha_k+\alpha_{k+1}}\right)
\end{align*}
Applying the bound $\log(1+x)\geqslant x-\frac{x^2}{2}$ and recalling that $\E(Y_k)=\frac{\alpha_k}{\alpha_k+\alpha_{k+1}}$, we find
\begin{align*}
I(u_k)&\geqslant -\frac{\alpha_k}{\alpha_k+\alpha_{k+1}}r^{k/2}C\log(N)+(C\log(N))^2\\&+\frac{\alpha_k}{\alpha_k+\alpha_{k+1}}r^{k/2}C\log(N)-\frac{\alpha_k}{2(\alpha_k+\alpha_{k+1})^2}r^k(C\log(N))^2\\
&\geqslant (C\log(N))^2\left(1-\frac{r}{2\alpha_0 (1+r)^2}\right)\;.
\end{align*}
Provided $C$ is large enough, this last quantity is larger than $c\log (N)$ uniformly over all parameters.
\end{proof}

Our last result concerns a tail estimate on the variables $\eta_k$ under $\pi_N$. We will use the notation
$$\nabla x_k :=x_{k+1}-x_{k-1}=\eta_{k}+\eta_{k+1}.$$

\begin{lem}\label{control du gradiant}
Fix $c>0$. There exist two constants $C',C'' > 0$ such that for all $N$ large enough, for all $\lambda \in [c/N,1-c]$ and for all $k\in \{1,\ldots,N-1\}$ we have $\pi_N(\nabla x_k\leqslant C'N^{-4}\lambda r^{k-N})\leqslant C''N^{-5}$.
\end{lem}

\begin{proof}
Under the measure $\pi_N$, we have $\nabla x_k / N\sim \beta(u,\alpha-u)$ with
\begin{align*}
u=\alpha_k+\alpha_{k+1}=\alpha_0 r^k(1+r)\;,\quad \alpha = \sum_{i=1}^N\alpha_i=\alpha_0 r\frac{r^N-1}{r-1}\;.
\end{align*}
Using the assumption on $\lambda$, a simple calculation shows that there exists $\varrho \in (0,1)$ and $C'>0$ such that $C' \lambda r^{k-N} \le u/(u+v) < \varrho$ for all $k\in \{1,\ldots,N-1\}$ and all $N\ge 2$. We apply Lemma \ref{absolue continuite} with $u=\alpha_k+\alpha_{k+1}$ and $v=\alpha-u$. We let $Z\sim\Gamma(u, u+v)$ and we compute for all $y>0$ and $t>0$
\begin{align*}
\pi_N(\nabla x_k\leqslant t)&\leqslant C_\varrho\, \P(NZ\leqslant t)\\
&\leqslant C_\varrho\,\P\left(e^{-NyZ}\geqslant e^{-yt}\right)\\
&\leqslant C_\varrho\, e^{yt}\E\left(e^{-NyZ}\right)\\
&\leqslant C_\varrho\, e^{yt}\left(1+\frac{Ny}{u+v}\right)^{-u}\;.
\end{align*}
Take $y=\frac{N^4(u+v)}{u}$ and $t=C'N^{-4}\lambda r^{k-N}$. Using the lower bound on $u/(u+v)$ together with the fact that $u \geqslant 1$ and the function $-u\log\left(1+\frac{N^5}{u}\right)$ is non increasing in $u$ on $\R^+$, we obtain
\begin{align*}
\pi_N(\nabla x_k\leqslant t)&\leqslant C_\varrho \exp\left(\frac{(u+v) C'\lambda r^{k-N}}{u}\right)\exp\left(-u\log\left(1+\frac{N^5}{u}\right)\right)\\
&\leqslant C_\varrho \frac{e}{1+N^5}\;.
\end{align*}
\end{proof}

\subsection{An explicit eigenvalue/eigenfunction of $\Lf$}

Let us provide an explicit eigenvalue of the generator $\Lf$, which is a self-adjoint operator on $L^2(S_N,\pi_N)$ (as a finite sum of orthogonal projectors).

\begin{lem}\label{Lemma:Eigen}
	Set
	$$ \gamma_N := -\left(1-\sqrt{1-\lambda^2}\cos\left(\frac{\pi}{N}\right)\right)\;,\quad f_N=\sum_{k=1}^{N-1}r^{-k/2}\sin\left(\frac{k\pi}{N}\right) g_k\;,$$
	where $g_k(x)=x_k-\overline{x}_k$ and $\overline{x}_k=N\frac{r^k-1}{r^N-1}$. For any $N\ge 2$, $\gamma_N$ is an eigenvalue of $\Lf$ associated to the eigenfunction $f_N$.
\end{lem}
It happens that $-\gamma_N$ is actually the spectral gap of $\Lf$, we refer to~\cite{Enguerand}.
\begin{proof}
Recall that $\rho_k$ is the $\beta(\alpha_k,\alpha_{k+1})$ distribution: we keep the notation $\rho_k$ to denote its density. Recall from Lemma \ref{esperance loi beta} that its mean equals $\alpha_k/(\alpha_k+\alpha_{k+1})$. We compute
\begin{align*}
\Lf g_k(x)&=\int_0^1\left((1-u)x_{k-1}+u x_{k+1}-x_k \right)\rho_k(u)\dx u\\
&=x_{k-1}\frac{\alpha_{k+1}}{\alpha_k+\alpha_{k+1}}+x_{k+1}\frac{\alpha_k}{\alpha_k+\alpha_{k+1}}-x_k\\
&=\frac{r}{1+r}x_{k-1}-x_k+\frac{1}{1+r}x_{k+1}\\
&=\frac{r}{1+r}g_{k-1}(x)-g_k(x)+\frac{1}{1+r}g_{k+1}(x),
\end{align*}
where we used the identity $\frac{r}{1+r}\overline{x}_{k-1} + \frac{1}{1+r}\overline{x}_{k+1} = \overline{x}_k$. So we have :
\begin{align*}
\Lf f_N&=\sum_{k=1}^{N-1}r^{-k/2}\sin\left(\frac{k\pi}{N}\right)\left(\frac{r}{1+r}g_{k-1}-g_k+\frac{1}{1+r}g_{k+1}\right)\\
&=\sum_{k=1}^{N-1}r^{-k/2}\sin\left(\frac{k\pi}{N}\right)\left(\frac{1}{2}(g_{k-1}-2g_k+g_{k+1})-\frac{\lambda}{2}(g_{k+1}-g_{k-1})\right)
\end{align*}
Note that $g_0 = g_N \equiv 0$. A simple computation then shows that this last term equals
\begin{align*}
\Lf f_N&=\sum_{k=1}^{N-1}g_k r^{-k/2}\Big[ \sin\left(\frac{k\pi}{N}\right)+\frac{1+\lambda}{2}\frac{1}{\sqrt{r}}\sin\left(\frac{(k+1)\pi}{N}\right)\\
&\qquad\qquad\qquad+\frac{1-\lambda}{2}\sqrt{r}\sin\left(\frac{(k-1)\pi}{N}\right) \Big]\\
&=-\left(1-\sqrt{1-\lambda^2}\cos\left(\frac{\pi}{N}\right)\right) f_N
\end{align*}
using the identities $\frac{1+\lambda}{2}\frac{1}{\sqrt{r}}=\frac{1-\lambda}{2}\sqrt{r}=\frac{\sqrt{1-\lambda^2}}{2}$.
\end{proof}

\subsection{A useful monotone coupling}\label{coupling}

In this subsection, we introduce a coupling between two versions $X^x$ and $X^y$ of our process that start from two configurations $x$ and $y$ that are ordered, that is, $y_i\le x_i$ for every $i\in \{0,\ldots,N\}$. The coupling preserves this order at all times. In addition, it maximizes the probability of merging the two interfaces at every resampling event.

For any $x\in S_N$ and any $k\in \{1,\ldots,N-1\}$, we define the interval $I(x,k)=[x_{k-1},x_{k+1}]$, we note that $\nabla x_k= \eta_k+\eta_{k+1} = |I(x,k)|$ is the length of this interval. Given a segment $[a,b]$, we let $\rho_k([a,b])$ be the distribution $\beta(\alpha_k,\alpha_{k+1})$ rescaled on the segment $[a,b]$ whose density $B_k[a,b]$ is given by
$$ B_k[a,b](u) = \frac{\Gamma(\alpha_k+\alpha_{k+1})}{\Gamma(\alpha_k)\Gamma(\alpha_{k+1})} \frac{(u-a)^{\alpha_k-1} (b-u)^{\alpha_{k+1}-1}}{(b-a)^{\alpha_k+\alpha_{k+1}-1}}\mathbf{1}_{[a,b]}(u)\;.$$
For simplicity we denote $I^x_k(t)=I(X^x(t),k)$.\\ 
Given two configurations $y \le x$ in $S_N$, we set
$$p(t,k)=1-\Vert \rho_k(I^y_k(t))-\rho_k(I^x_k(t))\Vert_{TV}.$$
We set $q(t,k)=1-p(t,k)$. We will abbreviate these into $p$ and $q$ when $t$ and $k$ are clear from the context. Then we define three probability measures $\nu_1(t,k),\nu_2(t,k),\nu_3(t,k)$ with respective densities $g_1,g_2,g_3$ given by 
\begin{align*}
g_1 &=q^{-1}(\rho_k(I^y_k)-\rho_k(I^x_k))_+,\\
g_2 &=p^{-1}\min(\rho_k(I^y_k),\rho_k(I^x_k)),\\
g_3 &=q^{-1}(\rho_k(I^x_k)-\rho_k(I^y_k))_+.
\end{align*}

\begin{figure}
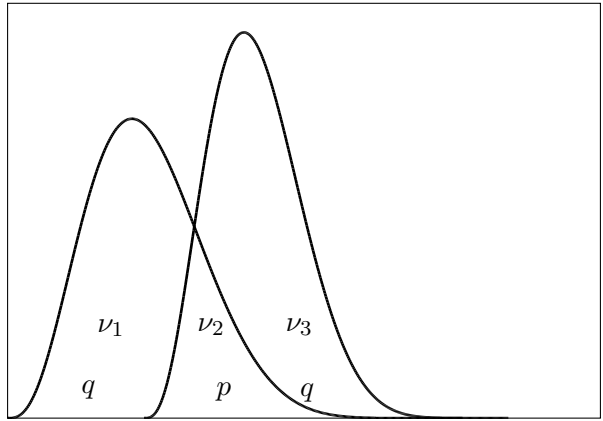

\caption{A plot of $B_k(I^x_k)$ and $B_k(I^y_k)$.}\label{Intersection coupling}
\end{figure}

Recall that $1\leqslant\alpha_k \leqslant\alpha_{k+1}$ for every $k\in \{1,\ldots,N-1\}$. We then claim that there is a unique point where $B_k(I^y_k)$ and $B_k(I^x_k)$ intersect, see Figure \ref{Intersection coupling} for an illustration. Existence is a consequence of continuity and of the fact that these two functions integrate to $1$. Regarding uniqueness, for $u\in I^y_k\cap I^x_k$, the equation $B_k(I^y_k)(u)=B_k(I^x_k)(u)$ yields
\begin{align*}
\left(\frac{u-y_{k-1}}{\nabla y_k}\right)^{\alpha_k-1}\left(\frac{y_{k+1}-u}{\nabla y_k}\right)^{\alpha_{k+1}-1}&=
\left(\frac{u-x_{k-1}}{\nabla x_k}\right)^{\alpha_k-1}\left(\frac{x_{k+1}-u}{\nabla x_k}\right)^{\alpha_{k+1}-1}
\end{align*}
which rewrites
\begin{align*}
\left(\frac{\nabla x_k}{\nabla y_k}\right)^{\alpha_k+\alpha_{k+1}-2}&=\left(\frac{u-x_{k-1}}{u-y_{k-1}}\right)^{\alpha_k-1}\left(\frac{x_{k+1}-u}{y_{k+1}-u}\right)^{\alpha_{k+1}-1}\;.
\end{align*}
Being the product of two increasing positive functions, the right hand side is itself increasing (here we use the inequalities $y_{k-1}\leqslant x_{k-1}$ and $y_{k+1}\leqslant x_{k+1}$). We conclude there is at most one solution of the equation, that we denote $s$.\\
Hence $g_1$ is suppported on $I_1=[y_{k-1},s]$ while $g_3$ is supported on $I_3=[s,x_{k+1}]$.\\

We now introduce our coupling. We consider a collection of independent, rate one Poisson processes indexed by $k\in \{1,\ldots,N-1\}$. If the $k$-th Poisson process rings at time $t$ then:
\begin{itemize}
\item with probability $p(t_-,k)$, we set $X^y_k(t)$ and $X^x_k(t)$ to a same random value drawn from the distribution $\nu_2(t_-,k)$,
\item with probability $q(t_-,k)$, we draw $X^y_k(t)$ according to the distribution $\nu_1(t_-,k)$ and independently we draw $X^x_k(t)$ according to the distribution $\nu_3(t_-,k)$.
\end{itemize}
It is clear that this coupling is order preserving.

\bigskip

The proof of the upper bound on the mixing time requires to estimate $q(t,k)$. To that end, we define $\delta X_k=X_k^x-X_k^y$ and we set
$$\overline{X}_k^x=\frac{(1-\lambda)X_{k+1}^x+(1+\lambda)X_{k-1}^x}{2}$$
which is nothing but the mean value of $X_k^x$ at a resampling event. Finally we define $\delta\overline{X}_k=\overline{X}_k^x-\overline{X}_k^y$.

\begin{lem}\label{controle de Q}
There is a constant $C > 0$ such that for all $\lambda$ in a compact set of $[0,1)$, all $t \ge 0$, all $k \in \{1,\ldots,N-1\}$ and all $N\ge 2$, we have:
$$q(t,k)\leqslant Cr^{k/2}Q(t,k)\;,$$
where
$$Q(t,k)=\frac{\delta\overline{X}_k}{\max(\nabla X_k^x(t_-),\nabla X_k^y(t_-))}.$$
\end{lem}
\begin{proof}
This lemma does not depend on our dynamics, it is a result on the beta distributions on two intervals $[\ell_1,r_1]$ and $[\ell_2,r_2]$ with $r_1\leqslant r_2$ and $\ell_1\leqslant \ell_2$. By symmetry (the arguments below remain true if one switches $\alpha_k$ and $\alpha_{k+1}$), we can assume that the first interval is larger than the second.  Moreover, the values of $q$ and $Q$ are invariant under any affine transformation of the coordinates so we can assume $\ell_1=0$ and $r_1=1$. We have 
$\ell_2\geqslant 0$ and $1\leqslant r_2\leqslant \ell_2+1$. Recall that $B_k([0,1])$ is the density of the beta distribution. We first prove that $q\leqslant\max(\Vert B_k([0,1])\Vert_\infty,1)\ell_2$. If $\ell_2\geqslant 1$, then the inequality is trivially satisfied, we suppose now $\ell_2<1$. We set $b=r_2-\ell_2$ and recall that $B_k([\ell_2,r_2])$ is the density of the beta distribution on the segment $[\ell_2,r_2]$.

\begin{figure}
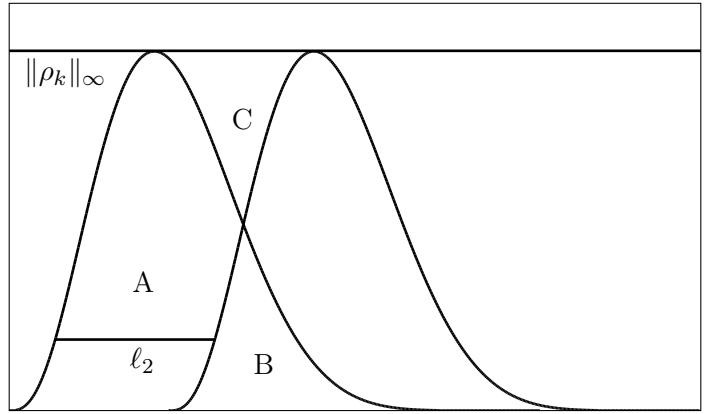


\caption{A plot of $B_k([0,1])$ and $B_k([\ell_2,\ell_2+1])$.}\label{Fig;s}
\end{figure}
We have 
\begin{align*}
q&=\int \Big(B_k([0,1])(u)-B_k([\ell_2,r_2])(u)\Big)_+\dx u\;.
\end{align*} 
A simple computation shows that $B_k([0,1])(u)-B_k([\ell_2,r_2])(u)$ which is explicitly given by
$$  \frac{\Gamma(\alpha_k+\alpha_{k+1})}{\Gamma(\alpha_k)\Gamma(\alpha_{k+1})} \Big( u^{\alpha_k-1}_+ (1-u)^{\alpha_{k+1}-1}_+ - \frac{(u-\ell_2)_+^{\alpha_k-1}(b+\ell_2-u)_+^{\alpha_{k+1}-1}}{b^{\alpha_k+\alpha_{k+1}-1}}\Big)\;,$$
is a non-decreasing function of $b$ whenever $u-\ell_2\leqslant b \frac{\alpha_k}{\alpha_k+\alpha_{k+1}-1}$. On the other hand, we know that the maximum of $u\mapsto B_k([\ell_2,r_2])(u)$ is achieved at $u-\ell_2 = b \frac{\alpha_k-1}{\alpha_k+\alpha_{k+1}-2}$, and that necessarily at that point $B_k([\ell_2,r_2])(u) > B_k([0,1])(u)$. Consequently the unique point $s$ where the two densities meet satisfies $s-\ell_2 \le b \frac{\alpha_k-1}{\alpha_k+\alpha_{k+1}-2} \le b \frac{\alpha_k}{\alpha_k+\alpha_{k+1}-1}$. Since we can restrict the interval of integration to $u \le s$ in the expression of $q$, we deduce that $q$ is non-decreasing with $b$.

It thus suffices to consider the case where $b=1$. We see on Figure \ref{Fig;s} that $q$ is the area of part A and we bound this area by the area of A$\cup$C, which is equal to $\Vert \rho_k\Vert_\infty \ell_2$. Now observe that $\ell_2=\frac{\max(\ell_2-\ell_1,r_2-r_1)}{\max(r_1-\ell_1,r_2-\ell_2)}\leqslant \frac{2}{1-\lambda} Q$. By Lemma \ref{esperance loi beta}, $\Vert \rho_k\Vert_\infty \le C_2 r^{k/2}$ and we can conclude.
\end{proof}

\section{The upper bound}\label{Sec:Upper}
Let $P_t^x$ be the law of the process at time $t$ starting from $x\in S_N$, and let us denote by $P_t^\pi$ the law of the process at time $t$ starting from equilibrium. For $\delta > 0$ we set
$$t_\delta := (1+\delta)\frac{\log\left(\frac{1+\lambda}{1-\lambda}\right)}{1-\sqrt{1-\lambda^2}} N\;.$$
Note that $t_\delta\sim\frac{4N}{\lambda}$ as $N\to\infty$, whenever $\lambda\to 0$. The upper bounds stated in Theorems \ref{Th:2} and \ref{Th:Main} are consequences of the following theorem, which is the main result of this section.

\begin{thm}\label{upper bound}
Assume that $\lambda\gg\frac{\log N}{N}$. For any $\delta > 0$, we have 
$$\sup_{x\in S_N}\Vert P_{t_\delta}^x-\pi\Vert_{TV}\too 0\;,\quad N\to\infty\;.$$
\end{thm}

By the triangle inequality and the stationarity of $\pi$, the theorem follows from the convergences
\begin{equation}\label{Eq:CV1}
\Vert P_{t_\delta}^{\max}-P_{t_\delta}^\pi\Vert_{TV}\too 0\;,
\end{equation}
and
\begin{equation}\label{Eq:CV2}
\sup_{x\in S_N}\Vert P_{t_\delta}^{\max} - P_{t_\delta}^x\Vert_{TV}\too 0\;.
\end{equation}
The rest of this section is devoted to proving them. The proof of \eqref{Eq:CV2} follows from the same arguments as the proof of \eqref{Eq:CV1} but requires \eqref{Eq:CV1} as an input. Let us start with the proof of \eqref{Eq:CV1}.\\

We work under the monotone coupling introduced in Subsection \ref{coupling} and let $(X^{\max},X^\pi)$ be the pair of processes defined under it (conditionally given the realisation $y$ of the initial condition of $X^\pi$, we apply the coupling with $x=\max$ and $y$). Since ${\max}=(N,...,N)$ is the maximal element of $S_N$ for the partial order on $\R^{N-1}$, the monotonicity of our coupling implies that almost surely for all $t\ge 0$, $X^{\max}(t) \ge X^\pi(t)$.\\
Recall that $f_N$ is the eigenfunction identified in Lemma \ref{Lemma:Eigen}. We now set
$$A_t=\sum_{k=1}^{N-1}r^{-k/2}\sin\left(\frac{k\pi}{N}\right)(X^{\max}_k(t)-X^\pi_k(t)) = f_N(X^{\max}(t)) - f_N(X^\pi(t))\;,$$
which can be seen as a twisted area between the two interfaces. Our proof of \eqref{Eq:CV1} consists of two steps
\begin{enumerate}[label=(\roman*)]
\item from the exponential decay of $t\mapsto \E(A_t)$, we deduce that $A_{t_{\delta}}$ is small with large probability (this step is independent of the coupling)
\item we show that, provided the area is small enough, it hits $0$ within a time of order $\log N$ with large probability.
\end{enumerate}
Altogether, this shows \eqref{Eq:CV1} with $t_\delta$ replaced by $t_{2\delta}$, but since $\delta$ is arbitrary this is enough to conclude. The next two subsections present these two steps. The third subsection then presents the proof of \eqref{Eq:CV2}.

\subsection{First step}

Set $\llbracket 1,N-1 \rrbracket := \{1,\ldots,N-1\}$. The goal of this step is to show that $\P(\A_1)\ge 1-N^{-4}$ for all $N$ large enough, where
$$\A_1=\left\{\forall k\in\llbracket 1,N-1 \rrbracket,\; X^{\max}_k(t_\delta)-X^\pi_k(t_\delta)\leqslant N^{-8}r^{k/2-N}\right\}\;.$$
By Lemma \ref{Lemma:Eigen}, for any $t\ge 0$
\begin{align*}
\E[A_t]&=\E\left[f_N(X^{\max}(t))\right]-\E\left[f_N(X^\pi(t))\right]\\
&=e^{\gamma_N t}\,\E\left[f_N(X^{\max}(0))\right]-e^{\gamma_N t}\,\E\left[f_N(X^\pi(0))\right]\\
&=e^{\gamma_N t}\,\E[A_0]\\
&\leqslant e^{\gamma_N t}N\frac{1-r^{-N/2}}{1-r^{-1/2}}.
\end{align*}
Recall that $r=(1+\lambda)/(1-\lambda)$. Since $\lambda\gg\frac{\log(N)}{N}$, we have $\gamma_N=-(1-\sqrt{1-\lambda^2})+O(N^{-2})$. Taking $t=t_\delta$, we get $e^{\gamma_N t_\delta}=e^{-(1+\delta)\log(r)N+O(N^{-1}\lambda^{-1})}\leqslant r^{-(1+\delta/2)N}$ for all $N$ large enough. Moreover, for all $N$ large enough we have
$$ N\frac{1-r^{-N/2}}{1-r^{-1/2}} \le N\frac{1}{1-r^{-1/2}} \le cN \lambda^{-1} \le cN^2\;,$$
which is negligible compared to $r^{\eps N}$ for any given $\eps > 0$. We thus deduce that for all $N$ large enough
$$\E[A_t]\leqslant r^{-(1+\delta/4)N}\;.$$
For $k \in \llbracket 1,N-1 \rrbracket$ we have by the Markov inequality
\begin{align*}
&\P(X^{\max}_k(t_\delta)-X^\pi_k(t_\delta)\geqslant N^{-8}r^{k/2-N})\\
&\leqslant N^8r^{-k/2+N} \E\left[X^{\max}_k(t_\delta)-X^\pi_k(t_\delta)\right]\\
&\leqslant N^8r^{-k/2+N}N\sin\left(\frac{k\pi}{N}\right)\E\left[X^{\max}_k(t_\delta)-X^\pi_k(t_\delta)\right]\\
&\leqslant N^9r^N\E[A_{t_\delta}]= o(N^{-5})
\end{align*}
and this suffices to deduce that $\P(\bar{\A_1}) < N^{-4}$ for all $N$ large enough.

\subsection{Second step}
In this second step, we use the specificities of the monotone coupling of Subsection \ref{coupling} during a time $t'=10\log N$. Under this coupling, we will say that an update is successful if the two updated r.v.~are set to the same value (which occurs with probability $p(t_-,k)$ following the notation from that subsection). We define the events:
\begin{align*}
\A_2&=\left\{\forall k, \text{the $k^{\text{th}}$ coord.~is updated at least once in }(t_\delta,t_\delta+t']\right\}\;,\\
\A_3(t_\delta+t)&=\left\{\text{Every update that occurs in $(t_\delta,t_\delta+t]$ is successful}\right\}\;,
\end{align*}
for all $t\ge 0$.
On the event $\A_2\cap\A_3(t_\delta+t')$, it is easy to check that $X^{\max}_k(t_\delta+t') = X^\pi_k(t_\delta+t')$ for all $k$. To conclude the proof of \eqref{Eq:CV1}, it suffices to show that $\P(\A_2)\too 1$ and $\P(\A_3(t_\delta+t'))\too 1$ when $N$ goes to $\infty$.

\bigskip

We begin with $\A_2$. Let $(\mathcal{P}_k(t),t\ge 0)$ be the Poisson process that counts the number of updates of the $k$-th coordinate from time $t_\delta$ on. We write: \begin{align*}
\P(\A_2)&=\P(\forall k, \mathcal{P}_k(t')\geqslant 1)=\left(1-e^{-10\log(N)}\right)^{N-1}\too 1\;.
\end{align*}

Let us now control $\P(\A_3(t_\delta+t'))$. We denote by $(\tau_i)_{i \ge 1}$ the update times, in the increasing order, from time $t_{\delta}$ on: this is nothing but $t_\delta$ plus the ordered sequence of times at which the Poisson processes $(\mathcal{P}_k(t),t\ge 0)$, $k\in\llbracket 1, N-1\rrbracket$, increase. We will give an upper bound on $\P(\overline{\A}_3(\tau_i))$ by induction on $i\ge 1$ and show that $\P(\overline{\A}_3(\tau_i))$ goes to 0 when $N$ goes to $\infty$ for all $i \ge 1$.\\
Recall that by Lemma \ref{control du gradiant} we have
$$\pi(\nabla x_k\leqslant C'N^{-4}\lambda r^{k-N})\leqslant C''N^{-5}\;.$$
Note that for every $i\ge 1$, $X^\pi(\tau_i)$ is distributed according to $\pi$. Consequently if we set
$$\A_4(t)=\{\forall k,\nabla X_k^\pi(t)\geqslant C'N^{-4}\lambda r^{k-N}\}\;,$$
then for every $i\ge 1$, $\P(\overline{\A_4(\tau_i)})\leqslant N\times C''N^{-5}=C''N^{-4}$ by union bound.\\
For convenience, we let $\tau_0 := t_\delta$ and we note that $\A_3(\tau_0)$ is trivially satisfied. We denote by $\F_t$ the natural filtration associated to the processes $(X^{\max},X^\pi)$ up to time $t$.\\
Now we fix $i\ge 0$ and we bound from above the probability of $\overline{\A}_3(\tau_{i+1})$. Note that on the event $\A_3(\tau_i)$, $\overline{\A}_3(\tau_{i+1})$ is achieved if and only if the update at time $\tau_{i+1}$ is \emph{not} successful. Applying Lemma \ref{controle de Q}, we obtain that on the event $\A_3(\tau_i)$
$$ \P(\overline{\A}_3(\tau_{i+1})|\F_{\tau_i}) \le \sum_{k=1}^{N-1} Cr^{k/2}\frac{\delta\overline{X}_k(\tau_i)}{\nabla X_k^\pi(\tau_i)}\;.$$
On $\A_4(\tau_i)$, it holds $\nabla X_k^\pi(\tau_i)\geqslant C'N^{-4}\lambda r^{k-N}$ and on $\A_1\cap\A_3(\tau_i)$ we have $\delta\overline{X}_k(\tau_i)\leqslant\delta\overline{X}_k(t_\delta)\leqslant N^{-8}r^{k/2-N}$. Therefore on the event $\A_3(\tau_i)\cap {\A}_4(\tau_i) \cap {\A}_1$ we further obtain
$$ \P(\overline{\A}_3(\tau_{i+1})|\F_{\tau_i}) \le \sum_{k=1}^{N-1} Cr^{k/2}\frac{N^{-8}r^{k/2-N}}{C'N^{-4}\lambda r^{k-N}} \leqslant \frac{C}{C'}N^{-3}\;.$$
Putting everything together we have shown that
\begin{align*}
\P( \overline{\A}_3(\tau_{i+1})\cap \A_3(\tau_i)) &= \E\big[\P(\overline{\A}_3(\tau_{i+1}) | \F_{\tau_i}) \1_{\A_3(\tau_i)}\big] \\
&\leqslant \E\bigg[\Big(\1_{\overline{\A}_4(\tau_i) \cup \overline{\A}_1} + \1_{{\A}_4(\tau_i) \cap {\A}_1}\P(\overline{\A}_3(\tau_{i+1})|\F_{\tau_i}) \Big) \1_{\A_3(\tau_i)}\bigg] \\
&\leqslant \E\bigg[ \Big(\1_{\overline{\A}_4(\tau_i) \cup \overline{\A}_1} + \1_{{\A}_4(\tau_i) \cap {\A}_1}\frac{C}{C'}N^{-3}\Big) \1_{\A_3(\tau_i)}\bigg]\\
&\le \P(\overline{\A}_4(\tau_i) \cup \overline{\A}_1) + \frac{C}{C'}N^{-3}\;.
\end{align*}
From our previous estimates on the probabilities of $\A_1$ and $\A_4$, we deduce the existence of $C>0$ such that
$$ \P( \overline{\A}_3(\tau_{i+1})\cap \A_3(\tau_i)) \le C N^{-3}\;.$$
To conclude, we note that
$$ \P(\overline{\A}_3(\tau_{i+1})) = \P({\A}_3(\tau_{i})\cap\overline{\A}_3(\tau_{i+1})) + \P(\overline{\A}_3(\tau_{i}))\;,$$
and a simple recursion yields for every $i\ge 0$
$$\P(\overline{\A}_3(\tau_{i})) \le C i N^{-3}\;.$$

Let $i_{\max}$ be the total number of updates that occur in the time interval $(t_\delta,t_\delta+t']$. With the above estimate and since $i_{\max}$ is independent from $(\overline{\A}_3(\tau_{i}))_{i\ge 0}$, we deduce that
$$ \P(\overline{\A_3}(t_\delta+t')) \le C N^{-1} \P(i_{\max} \le N^2) + \P(i_{\max} > N^2)\;.$$
Since $i_{\max}$ is a Poisson r.v.~of parameter $10 (N-1) \log N$, the r.h.s.~goes to $0$ as $N\to \infty$, thus proving the claim.
We conclude that for all $k$, we have $X_k^{\max}(t_\delta+t')=X_k^\pi(t_\delta+t')$ with high probability so that \eqref{Eq:CV1} follows.

\subsection{Proof of \eqref{Eq:CV2}}

We now couple $X^x$ with $(X^{\max}, X^{\pi})$ in such a way that $(X^{\max},X^x)$ follows the monotone coupling of Subsection \ref{coupling}. This may appear as a non-trivial extension of the coupling of Subsection \ref{coupling} to a triplet of interfaces, however we do not require any ordering between $X^x$ and $X^{\pi}$ and this makes the definition of this coupling rather straightfoward. Let us explain briefly how we proceed. The jump times of the three interfaces are given by the same Poisson clocks. Assume that the $k$-th Poisson clock rings, say at time $t$. Let $I^{\max}_k(t)$, resp.~$I^{\pi}_k(t)$, $I^{x}_k(t)$ be the interval formed by the two neighbours of $X^{\max}_k(t-)$, resp.~$X^{\pi}_k(t-)$, $X^{x}_k(t-)$. Recall that the density of $X^{\max}_k(t)$ is given by $\rho_k(I^{\max}_k(t))$, and set
\begin{align*}
	g^{\pi} &:= \frac{(\rho_k(I^{\pi}_k(t))-\rho_k(I^{\max}_k(t))_+}{\|\rho_k(I^{\pi}_k(t))-\rho_k(I^{\max}_k(t))\|_{TV}}\\
    g^{x} &:= \frac{(\rho_k(I^{x}_k(t))-\rho_k(I^{\max}_k(t))_+}{\|\rho_k(I^{x}_k(t))-\rho_k(I^{\max}_k(t))\|_{TV}}\;.
\end{align*}
Draw a uniform r.v.~$U$ over the bounded region of $\R^2$ that lies between the $x$-axis and the curve $\rho_k(I^{\max}_k(t))$. Independently, draw two r.v.~$Y^\pi$ and $Y^x$ according to the densities $g^{\pi}$ and $g^x$. Now set $X^{\max}_k(t)$ to the value obtained by taking the $x$-coordinate of $U$. To define $X^{\pi}_k(t)$ and $X^x_k(t)$ we argue according to the value of $U$.\\
If $U$ falls in the region formed by the $x$-axis and the curve $\rho_k(I^{\pi}_k(t))$, then we set $X^{\pi}_k(t) = X^{\max}_k(t)$. Otherwise, we set $X^{\pi}_k(t) = Y^\pi$. We proceed similarly for $X^x_k(t)$.\\
It is straightforward to check that this produces the desired coupling. We can now proceed with the proof of \eqref{Eq:CV2}. We take as an input \eqref{Eq:CV1}, and we follow the same strategy as in the previous steps.\\

We set $\A_5=\{\forall k, X_k^{\max}(t_\delta+t')=X_k^\pi(t_\delta+t')\}$, we have proven that $\P(\A_5)\too 1$. From now on, we work with the area between $X^{\max}$ and $X^x$:
$$A_t=\sum_{k=1}^{N-1}r^{-k/2}\sin\left(\frac{k\pi}{N}\right)(X^{\max}_k(t)-X^x_k(t)) = f_N(X^{\max}(t)) - f_N(X^x(t))\;.$$
We set
\begin{align*}
\A_1'&=\left\{\forall k, X^{\max}_k(t_\delta+t')-X^x_k(t_\delta+t')\leqslant N^{-8}r^{+k/2-N}\right\}\\
\A_2'&=\left\{\forall k, \text{the $k^{\text{th}}$ coord.~is updated at least once in }[t_\delta+t',t_\delta+2t']\right\}
\end{align*}
and for $t\in[0,t']$
\begin{align*}
\A_3'(t_\delta+t'+t)&=\left\{\text{Every update that occurs in $(t_\delta+t',t_\delta+t'+t]$ is successful}\right\}\;,
\end{align*}
On the event $\A_5\cap\A_1'\cap\A_2'\cap\A_3'(t_\delta+2t')$ we have $$X_k^x(t_\delta+2t')=X_k^{\max}(t_\delta+2t')$$ for all $k$. So it suffices to prove $\P(\A_i')\too 1$ when $N$ goes to infinity.\\
Regarding $\A_1'$ and $\A_2'$, this is exactly the same proof as before.\\
For $\A_3'(t_\delta+2t')$, it suffices to prove that $\P(\A_3'(t_\delta+2t')|\A_5)\too 1$ as $N$ goes to infinity. Similarly as before, we define the sequence of update times $(\tau_i')_i$, we introduce $\A_4'(t)=\{\forall k,\nabla X_k^{\max}(t)\geqslant C'N^{-4}\lambda r^{k-N}\}$. On the event $\A_5$, the interfaces $X^{\max}$ and $X^{\pi}$ are equal at all times $t\ge t_\delta+t'$, therefore the estimates on the invariant measure yield $\P(\overline{\A}_4'(t)|\A_5)\leqslant C''N^{-4}$. The previous proof then applies verbatim. This completes the proof of \eqref{Eq:CV2} (note that all the estimates are uniform over $x$).

%

\section{The Hamilton-Jacobi equation}\label{Sec:HJ}

We consider the Hamilton-Jacobi equation
\begin{equation}\label{Eq:HJ}
	\partial_t f+\partial_x f+(\partial_x f)^2=0\;,\quad (x,t)\in \Omega := (0,1)\times(0,+\infty)\;,
\end{equation}
with boundary condition prescribed on $\partial\Omega := \big((0,1]\times\{0\} \big)\cup \big(\{0\}\times (0,+\infty)\big) \cup \big(\{1\} \times (0,+\infty)\big)$ by
\begin{equation}\label{Eq:BC}
	\text{BC}(x,t)=\left\{\begin{array}{ll}
		0&\text{on }(0,1]\times\{0\}\;,\\
		1&\text{on }\{0\} \times (0,+\infty)\;,\\
		0&\text{on }\{1\} \times (0,+\infty)\;.
	\end{array}\right.
\end{equation}
Note that the point $(0,0)$ is excluded from the boundary, see Remark \ref{Rk:BC} for more details.

In this section, we prove existence and uniqueness of the solution of this equation. To that end, we need to recall some material from the theory of viscosity solutions introduced by Crandall and Lions~\cite{CrandallLions}.

\begin{defn}
We say that a function $u$ satisfies the viscosity inequality for the super-solution (respectively sub-solution) at a point $(x,t)$ if for any function $\phi$ that is $\mathcal{C}^\infty$ on a neighborhood of $(x,t)$ and such that $u-\phi$ admits a local minimum (respectively maximum) at $(x,t)$, we have the inequality
$$\partial_t \phi+\partial_x \phi+(\partial_x \phi)^2\geqslant 0\text{ (respectively }\leqslant)\;.$$
\end{defn}

The rigorous definition of being a solution to our equation relies on the notion of sub-solution and super-solution of viscosity. These definitions for non-continuous boundary conditions come from \cite{CIL}.
\begin{defn}[Viscosity solution]\label{definition viscosite}
Let $u$ be a lower semi-continuous function (respectively upper semi-continuous function) on $[0,1]\times[0,+\infty)$. We say that $u$ is a viscosity super-solution (respectively sub-solution) of \eqref{Eq:HJ} if:
\begin{enumerate}
\item for every $(x,t) \in \Omega$ it satisfies the viscosity inequality for the super-solution (resp.~sub-solution),
\item for every $(x,t) \in \partial \Omega$ it satisfies the viscosity inequality for the super-solution (resp.~sub-solution) or the boundary inequality $u(x,t)\geqslant$ BC$(x,t)$ (resp.~$u(x,t)\leqslant$ BC$(x,t)$).
\end{enumerate}
We say that $u$ is a viscosity solution if it is both a viscosity sub-solution and super-solution.
\end{defn}

Our next result shows uniqueness of viscosity solution under some relatively mild assumptions.
\begin{prop}\label{Th:Uniq}
Let $u$ be a viscosity sub-solution and $v$ a viscosity super-solution of \eqref{Eq:HJ}. Assume that
\begin{enumerate}
\item for all $(x,t) \in [0,1]\times[0,+\infty)$ we have
$$0\leqslant u(x,t)\leqslant 1-x\;,\quad \text{ and }\;0\leqslant v(x,t)\leqslant 1-x\;,$$
\item for all $t\ge 0$ the function $x\longmapsto u(x,t)$ is continuous at 0,
\item the function $t\longmapsto v(0,t)$ is continuous. 
\end{enumerate}
Then $u\leqslant v$ on $[0,1]\times[0,+\infty)$.
\end{prop}
\begin{rem}\label{Rk:BC}
We have not prescribed any boundary condition at $(0,0)$. Indeed uniqueness holds in this setting, and this is essentially due to the following two facts: (1) for any viscosity sub-solution $u$ that satisfies $0\leqslant u(x,t)\leqslant 1-x$ we have $u(0,0) = 0$, see Lemma \ref{Lemma:u}, and (2) we impose a priori $v(0,0) \ge 0$ to our viscosity super-solutions.\\
Note that a natural boundary condition at $(0,0)$ would be given by taking the upper and lower semi-continuous extensions of BC, thus yielding two different values for the sub-solutions and super-solutions.
\end{rem}

Before we proceed to the proof of the proposition, we first collect an a priori estimate.

\begin{lem}\label{Lemma:u}
Let $u$ be a viscosity sub-solution such that $0 \le u(x,t) \le 1-x$ for all $(x,t) \in [0,1]\times[0,4)$. Then $u(x,t) \le t/4$ for all $x\in (0,1]$ and all $t\in [0,4)$, and in particular $u(x,0) = 0$ for all $x\in (0,1]$.
\end{lem}

This result remains true if one replaces $[0,4)$ by $[0,T)$ for some arbitrary $T>0$. However, since $u$ is bounded from above by $1$, this bound is pointless for $T>4$.

\begin{proof}
Since $0\leqslant u(x,t)\leqslant 1-x$, necessarily we have $u(1,t)=0$. It remains to cover the remaining values of $x$.

We will use the function $\phi(x,t)=\frac{t}{4}+\frac{\sigma}{4-t}+\frac{K}{x}$, for some parameters $K>0$ and $\sigma>0$. that will be adjusted in the proof. We set $(x_0,t_0)\in\text{argmax}\{u-\phi\}$. Note that the latter exists and lies in $(0,1]\times [0,4)$. Since $\phi$ is $\mathcal{C}^\infty$ in $(0,1]\times [0,4)$, it is locally $\mathcal{C}^\infty$ around $(x_0,t_0)$. Simple computations show that, in $]0,1]\times [0,4[$, we have the inequalities $\partial_t \phi  > 1/4$ and $\partial_x\phi+(\partial_x\phi)^2 \geqslant -1/4$ so that
\begin{equation}\label{Eq:phiContrad}
\partial_t \phi+\partial_x\phi+(\partial_x\phi)^2>0\;.
\end{equation}
We now distinguish two cases.

\medskip

\textbf{First case: $t_0>0$}. Recall that $u(1,t) = 0$. Since $t\mapsto \phi(x,t)$ is increasing, we have $u(1,0)-\phi(1,0)>u(1,t_0)-\phi(1,t_0)$ and therefore we cannot have $x_0=1$. Consequently $(x_0,t_0)\in(0,1)\times(0,4)$ and $u$ must satisfy the viscosity inequality for the sub-solution at $(x_0,t_0)$, thus raising a contradiction with \eqref{Eq:phiContrad}.

\medskip

\textbf{Second case: $t_0=0$}. First assume that there exists $x$ such that $u(x,0)>0$. We claim that, provided $K$ is small enough, we have $x_0\in\{y>0, u(y,0)>0\}$. Indeed, take $y \in (0,1]$ such that $u(y,0)=0$. If $y < x$ then we have $u(x,0)-\frac{K}{x}>u(y,0)-\frac{K}{y}$ so that $x_0$ cannot lie in $\{ y \in (0,x): u(y,0) = 0\}$. If $y>x$ then, provided $K < \frac{u(x,0)}{\frac{1}{x}-1}$, we have $u(x,0)-\frac{K}{x}>u(y,0)-\frac{K}{y}$ thus concluding the proof of the claim. Consequently $u(x_0,0)>0$, so that $u$ does not satisfy the boundary inequality at $(x_0,0)$ and must satisfy the viscosity inequality for the sub-solution, thus raising a contradiction with \eqref{Eq:phiContrad}.
We have therefore proved that $u(x,0) = 0$ for all $x\in (0,1]$, and in particular, $u(x_0,0) = 0$. This means that $u(x,t) \le \phi(x,t)$ on $(0,1]\times[0,4)$. Taking $K,\sigma \too 0$ we get
$$u(x,t)\leqslant \frac{t}{4}$$ as required.
\end{proof}

We now proceed to the proof of our uniqueness result. This is an adaptation of the proof of \cite[Th 7.5]{CIL} to a setting where less regularity is assumed on the boundary condition and the super and sub-solutions.

\begin{proof}[Proof of Proposition \ref{Th:Uniq}]
Fix two parameters $\sigma>0$ and $T>0$, and suppose by contradiction that $M :=\max_{(x,t)\in[0,1]\times[0,T)}\left\{u(x,t)-v(x,t)-\frac{\sigma}{T-t}\right\}$ satisfies $M>0$. Let $(x_0,t_0)$ be an argument of this maximum. By the previous lemma, we know that $u(x,0)= 0$, and by hypothesis, we have $u(1,t) = 0$ and $v(x,t)\geqslant 0$. We thus deduce that $x_0\not=1$ and $t_0\not=0$. Necessarily either $(x_0,t_0)\in(0,1)\times(0,T)$ or $x_0=0$. We now distinguish two cases, that will both raise a contradiction.

\medskip

\textbf{First case:} Assume that all argmax $(x_0,t_0)$ of $M$ lie in $(0,1)\times(0,T)$. We set
$$M_\eps := \max_{(x,t), (y,s)\in[0,1]\times[0,T)}\left\{u(x,t)-v(y,s)-\frac{\sigma}{T-t}-\frac{(x-y)^2}{2\eps}-\frac{(t-s)^2}{2\eps}\right\}\;,$$
and we let $(x_\eps,t_\eps,y_\eps,s_\eps)$ be an argument of this maximum (argmax for brevity).

We claim that there is an argmax $(x_0,t_0)$ of $M$ for which $(x_{\eps_n},t_{\eps_n})\too (x_0,t_0)$ and $(y_{\eps_n},s_{\eps_n})\too (x_0,t_0)$ along some sequence $\eps_n \too 0$.

To prove the claim, we argue as follows. As a function of $\eps$, $M_\eps$ is non decreasing and $M_\eps\geqslant M$, so $M_\eps$ admits a limit, say $\alpha$, when $\eps$ goes to $0$. Moreover, we have
\begin{align*}
M_{2\eps} &\geqslant u(x_\eps,t_\eps)-v(y_\eps,s_\eps)-\frac{\sigma}{T-t_\eps}-\frac{(x_\eps-y_\eps)^2}{4\eps}-\frac{(t_\eps-s_\eps)^2}{4\eps}\\
&=M_\eps+\frac{(x_\eps-y_\eps)^2}{4\eps}+\frac{(t_\eps-s_\eps)^2}{4\eps}\;,
\end{align*}
so that $\frac{(x_\eps-y_\eps)^2}{4\eps}\too 0$ and $\frac{(t_\eps-s_\eps)^2}{4\eps}\too 0$ as $\eps\downarrow 0$. Since $(x_\eps,t_\eps)_\eps$ lies in a compact set, we can extract a sequence $(\eps_n)_n$ such that $(x_{\eps_n},t_{\eps_n})_n$ converges to some limit $(x',t')$. We have $(y_{\eps_n},s_{\eps_n})\too (x',t')$ and 
$$M\le \lim_n M_{\eps_n} = \alpha\leqslant u(x',t')-v(x',t')-\frac{\sigma}{T-t'}\leqslant M$$ 
because $u-v$ is u.s.c. We deduce that $(x',t')$ is an argmax of $M$, thus concluding the proof of the claim.

\medskip

Given the claim, for $n$ large enough $(x_{\eps_n},t_{\eps_n})$ and $(y_{\eps_n},s_{\eps_n})$ lie in $(0,1)\times(0,T)$. We set
$$\phi_u(x,t)=v(y_{\eps_n},s_{\eps_n})+\frac{\sigma}{T-t}+\frac{(x-y_{\eps_n})^2}{2\eps_n}+\frac{(t-s_{\eps_n})^2}{2\eps_n}.$$
By definition of $M_{\eps_n}$, $u-\phi_u$ admits a maximum at $(x_{\eps_n},t_{\eps_n})$. Since $u$ is a sub-solution we get
\begin{equation}\label{viscosite 1}
\partial_t\phi_u(x_{\eps_n},t_{\eps_n})+\partial_x\phi_u(x_{\eps_n},t_{\eps_n})+(\partial_x\phi_u(x_{\eps_n},t_{\eps_n}))^2\leqslant 0.
\end{equation}
In the same way, we set 
$$\phi_v(y,s)=u(x_{\eps_n},t_{\eps_n})-\frac{\sigma}{T-t_{\eps_n}}-\frac{(x_{\eps_n}-y)^2}{2\eps_n}-\frac{(t_{\eps_n}-s)^2}{2\eps_n},$$
the function $v-\phi_v$ admits a minimum at $(y_{\eps_n},s_{\eps_n})$ so that
\begin{equation}\label{viscosite 2}
\partial_t\phi_v(y_{\eps_n},s_{\eps_n})+\partial_x\phi_v(y_{\eps_n},s_{\eps_n})+(\partial_x\phi_v(y_{\eps_n},s_{\eps_n}))^2\geqslant 0.
\end{equation}
A computation shows that
\begin{align*}
\partial_t\phi_u(x_{\eps_n},t_{\eps_n})&=\frac{t_{\eps_n}-s_{\eps_n}}{\eps_n}+\frac{\sigma}{(T-t_{\eps_n})^2}\;,\quad &&\partial_t\phi_v(y_{\eps_n},s_{\eps_n})=\frac{t_{\eps_n}-s_{\eps_n}}{\eps_n}\\
\partial_x\phi_u(x_{\eps_n},t_{\eps_n})&=\frac{x_{\eps_n}-y_{\eps_n}}{\eps_n}\;,\quad &&\partial_x\phi_v(x_{\eps_n},t_{\eps_n})=\frac{x_{\eps_n}-y_{\eps_n}}{\eps_n}\;,
\end{align*}
which, plugged into (\ref{viscosite 1}) minus (\ref{viscosite 2}), yield the following contradiction
$$\frac{\sigma}{(T-t_{\eps_n})^2}\leqslant 0\;.$$

\medskip

\textbf{Second case:} Assume that there exists an argmax $(x_0,t_0)$ of $M$ such that $x_0=0$. We set
$$M_\eps := \max_{(x,t), (y,s)\in[0,1]\times[0,T)}\left\{u(x,t)-v(y,s)-\frac{\sigma}{T-t}-\left(\frac{x-y-\eps}{\eps}\right)^2-\frac{(t-s)^2}{2\eps}\right\}$$
and we let $(x_\eps,t_\eps,y_\eps,s_\eps)$ be an argument of this maximum.

\smallskip

We claim that there exists an argmax $(x',t')$ of $M$ for which $(x_{\eps_n},t_{\eps_n})\too (x',t')$, $(y_{\eps_n},s_{\eps_n})\too (x',t')$ and $x_{\eps_n}=y_{\eps_n}+\eps_n+o(\eps_n)$ along some sequence $\eps_n \too 0$.

\smallskip

To prove the claim, we argue as follows. Necessarily
\begin{align*}
u(x_\eps,t_\eps)-v(y_\eps,s_\eps)&-\frac{\sigma}{T-t_\eps}-\left(\frac{x_\eps-y_\eps-\eps}{\eps}\right)^2-\frac{(t_\eps-s_\eps)^2}{2\eps}\\
&\geqslant u(x_0+\eps,t_0)-v(x_0,t_0)-\frac{\sigma}{T-t_0}
\end{align*}
so that
\begin{align*}
\left(\frac{x_\eps-y_\eps-\eps}{\eps}\right)^2+\frac{(t_\eps-s_\eps)^2}{2\eps}\leqslant& u(x_\eps,t_\eps)-v(y_\eps,s_\eps)-(u(x_0+\eps,t_0)-v(x_0,t_0))\\
&-\frac{\sigma}{T-t_\eps}+\frac{\sigma}{T-t_0}.
\end{align*}
Since the right hand side is bounded from above uniformly over $\eps$, we deduce that $x_\eps-y_\eps-\eps\too 0$ and $t_\eps-s_\eps\too 0$ when $\eps$ goes to $0$. By compactness, there is a sequence $(\eps_n)_n$ such that $(x_{\eps_n},t_{\eps_n})$ converges and we denote by $(x',t')$ its limit. Necessarily $(y_{\eps_n},s_{\eps_n})\too(x',t')$. As $u-v$ is u.s.c. and $x\mapsto u(x,t_0)$ is continuous at $x_0=0$, we find
\begin{align*}
&\limsup_{\eps_n\to 0}\left(\frac{x_{\eps_n}-y_{\eps_n}-\eps_n}{\eps_n}\right)^2+\frac{(t_{\eps_n}-s_{\eps_n})^2}{2\eps_n}\\
&\leqslant u(x',t')-v(x',t')-(u(x_0,t_0)-v(x_0,t_0))-\frac{\sigma}{T-t'}+\frac{\sigma}{T-t_0} \leqslant 0.
\end{align*}
We deduce that $(x',t')$ is an argmax of $M$ and moreover $$\frac{x_{\eps_n}-y_{\eps_n}-\eps_n}{\eps_n}\too 0$$ so $x_{\eps_n}=y_{\eps_n}+\eps_n+o(\eps_n)$.

\bigskip

As a consequence of the claim and since $x' \ne 1$, for $n$ large enough, we have $(x_{\eps_n},t_{\eps_n})\in (0,1)\times(0,T)$ and $(y_{\eps_n},s_{\eps_n})\in [0,1)\times(0,T)$. Since $(x_{\eps_n},t_{\eps_n})$ does not lie on the boundary, $u$ satisfies the viscosity inequality for the sub-solution at $(x_{\eps_n},t_{\eps_n})$. If $x' \ne 0$, then for all $n$ large enough $(y_{\eps_n},s_{\eps_n})$ does not lie on the boundary, and $v$ satisfies the viscosity inequality for the super-solution at that point. On the other hand, if $x' = 0$ then necessarily $v(x',t') < 1$ (otherwise $(x',t')$ could not be an argmax) and by continuity of $t\mapsto v(0,t)$, we deduce that $v(x',t) < 1$ in a neighbourhood of $t'$, and therefore, for $n$ large enough, $v$ satisfies the viscosity inequality for the super-solution at $(y_{\eps_n},s_{\eps_n})$ even if $y_{\eps_n} = 0$.\\
We set
$$\phi_u(x,t)=v(y_{\eps_n},s_{\eps_n})+\frac{\sigma}{T-t}+\left(\frac{x-y_{\eps_n}-\eps_n}{\eps_n}\right)^2+\frac{(t-s_{\eps_n})^2}{2\eps_n}\;,$$
then $u-\phi_u$ admits a maximum at $(x_{\eps_n},t_{\eps_n})$ so that
\begin{equation}\label{viscosite 3}
\partial_t\phi_u(x_{\eps_n},t_{\eps_n})+\partial_x\phi_u(x_{\eps_n},t_{\eps_n})+(\partial_x\phi_u(x_{\eps_n},t_{\eps_n}))^2\leqslant 0.
\end{equation}
In the same way, we set
$$\phi_v(y,s)=u(x_{\eps_n},t_{\eps_n})-\frac{\sigma}{T-t_{\eps_n}}-\left(\frac{x_{\eps_n}-y-\eps_n}{\eps_n}\right)^2-\frac{(t_{\eps_n}-s)^2}{2\eps_n}\;,$$
then $v-\phi_v$ admits a minimum at $(y_{\eps_n},s_{\eps_n})$ and we have
\begin{equation}\label{viscosite 4}
\partial_t\phi_v(y_{\eps_n},s_{\eps_n})+\partial_x\phi_v(y_{\eps_n},s_{\eps_n})+(\partial_x\phi_v(y_{\eps_n},s_{\eps_n}))^2\geqslant 0.
\end{equation}
Simple computations show that
\begin{align*}
\partial_t\phi_u(x_{\eps_n},t_{\eps_n})&=\frac{t_{\eps_n}-s_{\eps_n}}{\eps_n}+\frac{\sigma}{(T-t_{\eps_n})^2}\;,\quad &&\partial_t\phi_v(y_{\eps_n},s_{\eps_n})=\frac{t_{\eps_n}-s_{\eps_n}}{\eps_n}\;,\\
\partial_x\phi_u(x_{\eps_n},t_{\eps_n})&=2\frac{x_{\eps_n}-y_{\eps_n}-\eps_n}{\eps_n^2}\;,\quad &&\partial_x\phi_v(x_{\eps_n},t_{\eps_n})=2\frac{x_{\eps_n}-y_{\eps_n}-\eps_n}{\eps_n^2}\;.
\end{align*}
Thus computing (\ref{viscosite 3}) minus (\ref{viscosite 4}), we get to the contradiction
$$\frac{\sigma}{(T-t_{\eps_n})^2}\leqslant 0\;.$$

We have therefore proven that $M\le 0$. Since this holds for all $\sigma > 0$, we deduce that $u\le v$.
\end{proof}

Now that we have proved uniqueness of viscosity solutions, we exhibit a solution using Lax's formula. To that end, we rewrite our equation in the generic form
$$\partial_t u+H(Du)=0\;,$$
with $H(p)=p+p^2$. We set $L(x,t;y,s) :=(t-s)H^*\left(\frac{x-y}{t-s}\right)$ for $s<t$ where $H^*$ is the convex dual of $H$ ie
$$H^*(r) :=\sup_{p\in\R}\{pr-H(p)\} =\frac{1}{4}(r-1)^2\;.$$
Lax's formula takes the form
$$S(x,t):=\inf_{\partial \Omega}\text{BC}(y,s)+L(x,t;y,s)=\left(\frac{1}{4t}\left((t-x)_+\right)^2\right)\wedge (1-x)\;.$$
Let us now prove that $S$ is both a sub-solution and a super-solution that satisfies the assumptions of Proposition \ref{Th:Uniq}.\\
We have $0\leqslant S\leqslant 1-x$, and at the boundary
$$S(x,0)=0\;,\quad S(1,t)=0\;,\quad S(0,t) = \min(\frac{t}{4},1)\;.$$
The boundary conditions are therefore satisfied, except for the super-solution on $\{0\}\times(0,4)$. However on this set we have $\partial_t S = 1/4$ and $\partial_x S+(\partial_x S)^2\geqslant \frac{-1}{4}$: consequently, $S$ satisfies the viscosity inequality for the super-solution on $\{0\}\times(0,4)$.\\
Let us now prove that $S$ satisfies the viscosity inequalities in $\Omega$. We remark that $S$ is also the solution of Lax's formula associated to the same PDE except that the boundary condition BC$(0,t)=1$ is replaced by BC'$(0,t)=\min(\frac{t}{4},1)$. We can apply~\cite[Theorem 11.1.v]{Lions} and deduce that $S$ is a viscosity solution of the same PDE but with the boundary conditions BC'. This ensures that $S$ satisfies the viscosity inequalities in $(0,1)\times(0,+\infty)$. We conclude  that $S$ is the unique solution of our equation.

\section{The lower bound}\label{Sec:Lower}

The goal of this section is to prove the lower bounds of Theorems \ref{Th:2} and \ref{Th:Main}. The lower bound for the former is a consequence of the first hydrodynamic limit stated as Proposition \ref{Prop:HydroNaive} in the introduction: its proof is given in Subsection \ref{Subsec:Naive}.\\
Most of the section is therefore devoted to the lower bound of Theorem \ref{Th:Main}, that relies on the hydrodynamic limit stated as Theorem \ref{Th:PDE}. As explained in the introduction, this hydrodynamic limit is built on the transformation
$$T(u)=\frac{-1}{N}\log_r\left(\frac{u}{a_N}+r^{-N}\right)\;,\quad u\in [0,N]\;,$$
with $a_N=\frac{N}{1-r^{-N}}$. For a configuration $x \in S_N$, we write $z = T(x)$ as a shortcut for
$$ z_k = T(x_k)\;,\quad \forall k \in \{0,\ldots,N\}\;.$$

Note that the image of the maximal configuration $\max := (0,N,\ldots,N) \in S_N$ through $T$ is given by $(1,0,\ldots,0)$, while the minimal configuration $(0,\ldots,0,N) \in S_N$ is sent on $(1,\ldots,1,0)$. Note also that if $x\le x'$ then $T(x) \ge T(x')$ for the natural partial order on $\R^{N+1}$. Note also that $T(\pi_N(x_k)) \sim 1- \frac{k}{N}$.\\
%
%
%

%

%

Given the hydrodynamic limit of Theorem \ref{Th:PDE}, we can complete the proof of Theorem \ref{Th:Main}

\begin{proof}[Proof of Theorem \ref{Th:Main}]
	We already proved the upper bound in Section \ref{Sec:Upper}. It remains to show that, for any given $t<4$, we have
	$$ \| P^{\max}_{t N / \lambda} - \pi_N \|_{TV} \to 1\;,\quad N\to\infty\;.$$
	Fix $t<4$. Observe that $S(0,t)=\frac{t}{4}<1$. Since $S$ is a continuous function, there exists $y>0$ such that $S(y,t)<1-y$.
	We can thus fix $\eps>0$ such that $S(y,t)+\eps<1-y-\eps$. Set $k_N=\lfloor yN\rfloor$, and introduce $\A_N=\{x\in S_N: T(x_{k_N})\leqslant 1-y-\eps\}$. By Theorem \ref{Th:PDE}, we have $\P(X^{\max}(tN/\lambda)\in \A_N)\too 1$ when $N$ goes to infinity. To conclude, it suffices to show that $\pi_N(\A_N)\too 0$ when $N$ goes to infinity.\\
	Recall from Lemma \ref{Lem:InvMeas} that under $\pi_N$, the r.v.~$x_{k_N}/N$ is distributed according to $\beta(u,v)$ with $u=\sum_{i=1}^{k_N} \alpha_i$ and $v=\sum_{i=k_N+1}^N \alpha_i$. Recall also that $\alpha_i = \alpha_0 r^i$ and that $r = \frac{1+\lambda}{1-\lambda}$. As $N\to\infty$ we compute
	\begin{align*}
		\pi_N(x_{k_N})&=N\frac{u}{u+v}
		=N \frac{r^{k_N+1}-r}{r^{N+1}-r}
		\sim Nr^{k_N-N}\\
		\Var_{\pi_N}(x_{k_N})&=N^2\frac{uv}{\left(u+v\right)^2\left(u+v+1\right)}=O\left(N^2 r^{k_N-2N}\right).
	\end{align*}
	We deduce that under $\pi_N$ the r.v.~$x_{k_N}$ is smaller than $N r^{k_N-N+N\eps/2}$ with high probability so we have $\pi_N(\A_N)\too 0$ when $N$ goes to infinity.
\end{proof}

The proof of Theorem \ref{Th:PDE} is decomposed into three subsections :
\begin{enumerate}
	\item In Subection \ref{Subsec:M}, we introduce a Markov process $M$, which is a variant of our dynamics that satisfies $M\le X^{\max}$ with large probability. Then, we show that, provided $T(\E[X^{\max}])$ and $\E[T(M)]$ go to the solution of the Hamilton-Jacobi equation \eqref{Eq:HJ}, the hydrodynamic limit holds.
	\item In Subsection \ref{Subsec:ODE}, we show that $T(\E[X^{\max}])$ and (an upper bound of) $\E[T(M)]$ are solutions of simple systems of ODEs.
	\item In Subsection \ref{Subsec:CVODE}, we show that these solutions converge to $S$.
\end{enumerate}

Finally, in Subsection \ref{Subsec:Naive} we will provide the proof of Proposition \ref{Prop:HydroNaive}. We now work under the asumption that $\log N / N \ll \lambda \ll 1$, except in the last subsection where we only impose $1/N \ll \lambda$.

\subsection{Estimation of the transformed process}\label{Subsec:M}

An important ingredient of our proof is a Markov process $M$ that behaves very much like $X^{\max}$, but is built in such a way that it lies below $X^{\max}$ with large probability. In a nutshell: the processes $M$ and $X^{\max}$ have the same Poisson processes for their resampling times, and at each resampling event, $M$ picks the same location $k$ as $X^{\max}$ but instead of using a random resampling variable it relies on a deterministic resampling variable that is typically much smaller than the beta r.v.~used for $X^{\max}$.\\
To introduce precisely the process, we need some preliminary estimates on the resampling variables. Let $U_k$ be a $\beta(\alpha_k,\alpha_{k+1})$ r.v. Following Lemmas \ref{esperance loi beta} and \ref{concetration beta}, there exists a constant $C>0$ such that for all $k\in \llbracket 1,N\rrbracket$, one has
\begin{equation}\label{Eq:UkVar}
\P\left(U_k \le \frac{1-\lambda}2-C\log(N)r^{-k/2}\right)\leqslant N^{-5}.
\end{equation}
For every $k\in \llbracket 1,N\rrbracket$, we set
$$\mu_k:=\begin{cases}
1-\lambda&\text{if }k < k_0\\
2C\log(N)r^{-k/2}&\text{else}
\end{cases}\;,\quad \text{with }k_0 := \Big\lfloor N\sqrt{\frac{\log N}{\lambda N}} \Big\rfloor\;.$$
Since $\lambda\gg\frac{\log(N)}{N}$, it holds $\sqrt{\frac{\log(N)}{\lambda N}}\too 0$ and we deduce that for $k \ge k_0$ we have $\mu_k \ll \lambda^2$.\\
The process $M$ is then defined as follows. Initially we take
$$M_k(0) := \begin{cases} 0 \text{ if } k < k_0\\1 \text{ else}\;.
\end{cases}$$
If $t\in \R_+$ is a resampling time of $X^{\max}$, we pick the same update site $k$ as for $X^{\max}$ but we twist the resampling mechanism in the following way: $M_k(t)$ is set to the value 
$$\frac{1+\lambda+\mu_k}{2} M_{k-1}(t_-)+\frac{1-\lambda-\mu_k}{2} M_{k+1}(t_-)\;.$$

\begin{lem}\label{Lemma:MY}
As $N\to\infty$, it holds $$\P\big(\forall t\in [0,N^2], \forall k\in\llbracket 1,N-1\rrbracket,\quad M_k(t)\leqslant X^{\max}_k(t)\big)\too 1\;.$$
\end{lem}
Note that $N^2 \gg N / \lambda$ provided $\lambda \gg \log N / N$. Consequently, the comparison between $M$ and $X^{\max}$ holds on the time-scale at which we aim at proving the hydrodynamic limit.
\begin{proof}
We denote $(\tau_i)_{i=0}^\ell$ the resampling times for $M$ (and $X^{\max}$) on the time-interval $[0,N^2]$. By induction, for all $i\in \llbracket 1,\ell\rrbracket$ and all $k\in\llbracket 1,N-1\rrbracket$, we have $M_k(\tau_i)\leqslant X^{\max}_k(\tau_i)$ provided we work on the event
\begin{align*}
	\A &=\bigcap_k\Big\{\text{All the resampl.~var.~for $X^{\max}_k$ before time $N^2$ are larger than }\\
		&\qquad\qquad \frac{1-\lambda-\mu_k}{2}\Big\}.
\end{align*}
Combining with the fact $\ell\leqslant N^4$ with high probability, Equation \eqref{Eq:UkVar} and the definition of the $\mu_k$'s, we easily deduce that $\P(\A) \too 1$ as $N\to \infty$.
\end{proof}

Let us introduce the transformed processes
$$ T^X_k(t) := T(X^{\max}_k(t))\;,\quad T^M_k(t) := T(M_k(t))\;.$$
Note that, by Lemma \ref{Lemma:MY} and since the map $T$ is decreasing, $T^X \le T^M$ with large probability on the time-scale we are interested in. Our goal is to show that $T^X$ converges to the solution of \eqref{Eq:HJ}. It turns out that the process $T^X$ is rather difficult to analyze directly. We will instead control the expectation of $T^M$, and the image by $T$ of the expectation of $X^{\max}$. This is the content of the following result, whose proof relies on the convergence of explicit systems of ODEs towards the solution of the Hamilton-Jacobi equation, and is postponed to the next two subsections.

\begin{prop}\label{Prop:CVschemes}
For any $(x,t) \in (0,1]\times [0,+\infty)$, $T(\E[X^{\max}_{\lfloor xN\rfloor}(tN/\lambda)])$ and $\E[T^M_{\lfloor xN\rfloor}(tN/\lambda)]$ converge (locally uniformly) to $S(x,t)$.
\end{prop}

Combining this proposition with some comparisons between the processes $T^X$ and $T^M$, we can complete the proof of the hydrodynamic limit.

\begin{proof}[Proof of Theorem \ref{Th:PDE}]
Fix $x\in (0,1]$, $t> 0$ and set $k :=\lfloor xN \rfloor$. We will show that for any $\eps> 0$, as $N\too\infty$
\begin{equation}\label{Eq:ToProveYk}
	\P\Big(T(\E[X^{\max}_k(tN/\lambda)]) - \eps \leqslant T^X_k(tN/\lambda) \leqslant T(\E[X^{\max}_k(tN/\lambda)]) + \eps \Big) \too 1\;.
\end{equation}
Combined with the local uniform convergence stated in Proposition \ref{Prop:CVschemes}, this yields the following convergence in probability
$$ T^X_{k}(tN/\lambda)-S(x,t)\too 0\;,\quad N\to\infty\;.$$
The monotonicity in $k$ of $T^X_k$ and the continuity of $x\mapsto S(x,t)$ suffice to strengthen it into a convergence in probability which is uniform in $x$ over compact sets of $(0,1]$. We are left with the proof of \eqref{Eq:ToProveYk}. For simplicity, set $t':= tN/\lambda$.

First of all, for any $\eps > 0$
\begin{align*}
&\P\Big(T^X_k(t')\leqslant T\big(\E[X^{\max}_k(t')]\big) - \eps\Big)\\
&=\P\Big(X^{\max}_k(t')\geqslant T^{-1}\Big(T\big(\E[X^{\max}_k(t')]\big)-\eps \Big)\Big)\\
&=\P\left(X^{\max}_k(t')\geqslant a_N\left(\left(\frac{\E[X^{\max}_k(t')]}{a_N}+r^{-N}\right)r^{N\eps}-r^{-N}\right)\right)\\
&\leqslant\P\left(X^{\max}_k(t')\geqslant \E[X^{\max}_k(t')] r^{N\eps}\right)\\
&\leqslant r^{-N\eps}\;,
\end{align*}
which goes to $0$ as $N\to \infty$ since $\lambda \gg \log N / N$.

\medskip

Second, recall that $T$ is non-increasing. By Lemma \ref{Lemma:MY}, we deduce that
$$ \P(T^M_k(t') \geqslant T^X_k(t') ) \too 1\;,\quad N\to\infty\;.$$
Since $T$ is valued in $[0,1]$, we thus deduce that for all $\eps > 0$ and all $N$ large enough
\begin{equation}\label{Eq:TYM}
\E[T^X_k(t')] \leqslant \E[T^M_k(t')] + \eps\;.
\end{equation}
We have the following lower bound for any $0 < \eps < \eta$
\begin{align*}
\E[T^X_k(t')] &\geqslant  \P\big(T^X_k(t') \geqslant T(\E[X^{\max}_k(t')]) - \eps\big) (T(\E[X^{\max}_k(t')] - \eps)\\
&+ \P\big(T^X_k(t') \geqslant T(\E[X^{\max}_k(t')]) + \eta\big) (\eta+\eps)\;,
\end{align*}
which, combined with \eqref{Eq:TYM}, ensures that for all $N$ large enough
\begin{align*}
&\P\big(T^X_k(t') \geqslant T(\E[X^{\max}_k(t')]) + \eta\big)\\
&\leqslant \frac{\E[T^M_k(t')] + \eps - \P\big(T^X_k(t') \geqslant T(\E[X^{\max}_k(t')]) - \eps\big) (T(\E[X^{\max}_k(t')] - \eps)}{\eta + \eps}\;.
\end{align*}
Applying Proposition \ref{Prop:CVschemes} and the convergence already proven in the first part of the proof, we further deduce that for all $N$ large enough
\begin{align*} \P\big(T^X_k(t') \geqslant T(\E[X^{\max}_k(t')]) + \eta\big) &\leqslant \frac{S(x,t) + 2 \eps - (1-\eps)(S(x,t)-2\eps)}{\eta + \eps}\\
&\leqslant \frac{5 \eps}{\eta+\eps}\;.\end{align*}
Passing to the limit $\eps \downarrow 0$, we deduce that for any $\eta > 0$
$$\P\big(T^X_k(t') \geqslant T(\E[X^{\max}_k(t')]) + \eta\big) \too 0\;,\quad N\to\infty\;,$$
thus concluding the proof of \eqref{Eq:ToProveYk}.
\end{proof}

\subsection{Definitions and properties of the systems of ODEs}\label{Subsec:ODE}

This subsection and the next aim at proving Proposition \ref{Prop:CVschemes}. First we identify the systems of ODEs solved by $T(\E[X^{\max}])$ and (an upper bound of) $\E[T(M)]$. These systems of ODEs will look like discrete-space approximations of the Hamilton-Jacobi equation \eqref{Eq:HJ} on the set
$$ \Omega_N := \left\{\Big(\frac{k}{N},t\Big): k\in\llbracket 0,N\rrbracket\;,\; t\in\R_+\right\}\;.$$
Let us start with the systems of ODEs associated to $X^{\max}$. We set for $x,y,z \in [0,1]$
$$ H_X^N(x,y,z):= \frac{1}{\lambda\log(r)}\left(\frac{1+\lambda}{2}r^{N(y-x)}+\frac{1-\lambda}{2}r^{N(y-z)}-1\right)\;.$$

\begin{lem}
Set $f_X^N(x,t) := T(\E[X^{\max}_{xN}(tN/\lambda)])$ for all $(x,t) \in \Omega_N$. Then $f_X^N$ is the unique solution of 
$$\partial_t f_X^N(x,t)+H_X^N\left(f_X^N(x-\frac1{N},t),f_X^N(x,t),f_X^N(x+\frac1{N},t)\right) = 0\;, (x,t) \in \Omega_N\;,$$
endowed with the boundary conditions
\begin{align*}
f_X^N(0,t)=1\;,\quad f_X^N(1,t)=0\;,\quad f_X^N(x,0)=0\;.
\end{align*}
\end{lem}

\begin{proof}
The fact that the evolution equation of the statement admits a unique solution is a consequence of Cauchy-Lipschitz theory. Let us check that $f^N_X$ satisfies this evolution equation. Regarding the boundary conditions, this is a direct consequence of our choice of initial condition, of the boundary conditions imposed to $X^{\max}$ and of the definition of $T$. Regarding the evolution in time, we have (abbreviating $f_X^N$ in $f$)
\begin{align*}
&\partial_t T^{-1}f(x,t)=\partial_t\E[X^{\max}_{xN}(tN/\lambda)]\\
&=\frac{N}{\lambda}\left(\frac{1+\lambda}{2}\E[X^{\max}_{xN-1}(tN/\lambda)]+\frac{1-\lambda}{2}\E[X^{\max}_{xN+1}(tN/\lambda)]-\E[X^{\max}_{xN}(tN/\lambda)]\right)\;.
\end{align*}
Consequently, 
\begin{align*}
&\partial_t f(x,t)\\
&=T'(T^{-1}(f(x,t)))\partial_t( T^{-1}f(x,t))\\
&=\frac{-r^{Nf(x,t)}}{N\log(r)a_N}\frac{N}{\lambda}\left(\frac{1+\lambda}{2}T^{-1}f(x-1/N,t)+\frac{1-\lambda}{2}T^{-1}f(x+1/N,t)-T^{-1}f(x,t)\right)\\
&=\frac{-1}{\lambda\log(r)}\left(\frac{1+\lambda}{2}r^{N(f(x,t)-f(x-1/N,t))}+\frac{1-\lambda}{2}r^{f(x,t)-f(x+1/N,t)}-1\right)\;.
\end{align*}
\end{proof}

We pass to $f_M^N$. For any $x,y,z \in [0,1]$ and any $k\in \llbracket 1,N\rrbracket$, we set
$$ H_M^N(x,y,z,k)=\frac{1}{\lambda}\log_r\left(\frac{1+\lambda+\mu_k}{2}r^{N(y-x)}+\frac{1-\lambda-\mu_k}{2}r^{N(y-z)}\right)\;.$$

\begin{lem}
Denote by $f_M^N$ the unique solution of
$$\partial_t f_M^N(x,t) + H_M^N\left(f_M^N(x-1/N,t),f_M^N(x,t),f_M^N(x+1/N,t),xN\right) = 0\;,\quad (x,t) \in \Omega_N\;,$$
endowed with the boundary conditions
\begin{align*}
f_M^N(0,t)=1\;,\quad f_M^N(1,t)=0\;,\quad f_M^N(x,0)=\begin{cases}
1&\text{if } x<\sqrt{\frac{\log(N)}{\lambda N}}\\
0&\text{if } x\ge \sqrt{\frac{\log(N)}{\lambda N}}
\end{cases}
\end{align*}
Then for all $(x,t) \in \Omega_N$, $f_M^N(x,t) \geqslant \E\big[T^M_{xN}(tN / \lambda )\big]$.
\end{lem}
We will see in the proof of Proposition \ref{Prop:CVschemes} that the convergence of $f^N_X,f^N_M$ towards $S$, the inequality $M\le X^{\max}$ that holds with large probability, and some elementary arguments, allow to deduce the convergence of $\E[T^M]$ itself.
\begin{proof}
By definitions of the process $M$ and of the map $T$, the asserted inequality is actually an equality at time $0$ and at $x\in \{0,1\}$. Regarding the evolution in time, we argue as follows. For any given $(\alpha,\beta)\in[0,1]\times\R_+^*$, the map $h(x,y)=-\frac{1}{\beta}\log\left(\alpha e^{-\beta x}+(1-\alpha)e^{-\beta y}\right)$ is concave (its Hessian is non-positive). Taking $\beta = N\log(r)$, we observe that
$$T(\alpha T^{-1}(x)+(1-\alpha)T^{-1}(y))=h(x,y)\;.$$
so that, taking further $\alpha = (1+\lambda+\mu_k)/2$, we deduce from Jensen's inequality
\begin{align*}
&\E\left[T\left(\frac{1+\lambda+\mu_k}{2} M_{k-1}(tN/\lambda)+\frac{1-\lambda-\mu_k}{2} M_{k+1}(tN/\lambda)\right)\right]\\
&\le T\left(\frac{1+\lambda+\mu_k}{2} T^{-1}(\E[T^M_{k-1}(tN/\lambda)])+\frac{1-\lambda-\mu_k}{2} T^{-1}(\E[T^M_{k+1}(tN/\lambda)])\right)\;.
\end{align*}
Consequently a simple computation shows that
\begin{align*}
\partial_t\E\big[T^M_k(tN/\lambda)\big]&=\frac{N}{\lambda}\Big(\E\left[T\left(\frac{1+\lambda+\mu_k}{2} M_{k-1}(tN/\lambda)+\frac{1-\lambda-\mu_k}{2} M_{k+1}(tN/\lambda)\right)\right]\\&\quad-\E\big[T^M_k(tN/\lambda)\big]\Big)\\
&\leqslant -\frac{1}{\lambda}\log_r\Big(\frac{1+\lambda+\mu_k}{2}r^{-N\E[T^M_{k-1}(tN/\lambda)]+N\E[T^M_k(tN/\lambda)]}\\
&\qquad\qquad +\frac{1-\lambda-\mu_k}{2}r^{-N\E[T^M_{k+1}(tN/\lambda)]+N\E[T^M_k(tN/\lambda)]}\Big)\\
&\leqslant -H_M^N\left(T^M_{k-1}(tN/\lambda),T^M_{k}(tN/\lambda),T^M_{k+1}(tN/\lambda),k\right)\;.
\end{align*}
According to Definition \ref{defSol} given in the next subsection, we remark $\E\big[T^M_k(tN/\lambda)\big]$ is a subsolution for the system associated to $H_M^N$. We conclude by Proposition \ref{Prop:CompareScheme}.
\end{proof}

Let us now collect some useful properties of the systems.
\begin{lem}\label{increasing}
The maps $(x,y,z)\mapsto H_X^N(x,y,z)$ and, for any given $k$, $(x,y,z)\mapsto H_M^N(x,y,z,k)$ are non-increasing in their variables $x$ and $z$. In addition, we have for all $x,y,z,a$
$$ H_X^N(x+a,y+a,z+a) =H_X^N(x,y,z)\;,\quad H_M^N(x+a,y+a,z+a,k) = H_M^N(x,y,z,k)\;.$$
Finally, the functions $t\mapsto f_X^N(x,t)$ and $t\mapsto f_M^N(x,t)$ are non-decreasing.
\end{lem}
\begin{proof}
The stated properties on $H_X^N$ and $H_M^N$ are immediate from the definitions of these maps. We turn to the monotonicity of $t\mapsto f_X^N(x,t)$ and $t\mapsto f_M^N(x,t)$, and restrict ourselves to the former since the arguments are identical for the latter. We start with the following observation. Set $i\ge 1$, $t_0\ge 0$ and assume that for some set $C=\{k_0,k_0+1,\ldots,\ell_0\}$ we have $\partial_t^\ell f_X^N(k/N,t_0) = 0$ for every $k\in C$ and every $1\le \ell\le i-1$, $\partial_t^i f_X^N(k/N,t_0) > 0$ for $k=k_0, \ell_0$ and $\partial_t^i f_X^N(k/N,t_0) = 0$ for $k \in\{k_0+1,\ldots,\ell_0-1\}$. Then for all $k\in \{k_0+1,\ldots,\ell_0-1\}$ (a priori this expression contains many more terms but they all vanish)
\begin{align*}
	&\partial_t^{i+1}  f_X^N(\frac{k}{N},t_0)\\
	&= -\partial_x H_X^N(f_X^N(\frac{k-1}{N},t_0),f_X^N(\frac{k}{N},t_0),f_X^N(\frac{k+1}{N},t_0))\; \partial_t^{i} f_X^N(\frac{k-1}{N},t_0)\\
	&- \partial_z H_X^N(f_X^N(\frac{k-1}{N},t_0),f_X^N(\frac{k}{N},t_0),f_X^N(\frac{k+1}{N},t_0)) \;\partial_t^{i} f_X^N(\frac{k+1}{N},t_0)\;,
\end{align*}
 so that $\partial_t^{i+1}  f_X^N(k/N,t_0) > 0$ for $k = k_0+1,\ell_0-1$ and $\partial_t^{i+1}  f_X^N(k/N,t_0) = 0$ for $k\in \{k_0+2,\ldots,\ell_0-2\}$.\\
Let $t_0 := \inf\{t\ge 0: \exists k, \partial_t f_X^N(k/N,t) < 0\}$ and assume it is finite. Necessarily at time $t_0$, all $\partial_t f_X^N(k/N,t)$ are either positive or null. If all these derivatives are null, then $f_X^N$ remains constant in time after $t_0$ (indeed, it solves the equation that admits a unique solution by Cauchy-Lipschitz Theorem) and this raises a contradiction with the definition of $t_0$. Otherwise, take $i=1$ and let $C = \{k_0,\ldots,\ell_0\}$ be a set of sites at which these derivatives vanish except at $k_0,\ell_0$. Then applying successively the observation above, we deduce that for every $k\in C$, the first non-zero higher derivative of $f_X^N(k/N,t_0)$ is positive. This contradicts the definition of $t_0$.
\end{proof}

We prove that both $H^N_X$ and $H^N_M$ ``converge'' to $H(x) = x+x^2$, the Hamiltonian of our PDE.

\begin{prop}\label{Prop:PDE}
Fix $(x,t) \in \Omega$ and some function $\phi$ that is $\mathcal{C}^\infty$ in a neighbourhood of $(x,t)$. Let $(x_N,t_N)_N$ be a sequence of points in $\Omega$ that converges to $(x,t)$. Then
\begin{align*}
H_X^N(\phi(x_N-1/N,t),\phi(x_N,t),\phi(x_N+1/N,t))\too\partial_x\phi(x,t)+(\partial_x\phi(x,t))^2\;.
\end{align*}
The same convergence holds with $H^N_X$ replaced by $H^N_M$ provided we further assume that  $x_N\ge k_0/N$.
\end{prop}
\begin{proof}
The regularity of $\phi$ near $(x,t)$ ensures that as $N\to\infty$ 
\begin{align*}
&N\big(\phi(x_N-1/N,t_N)-\phi(x_N,t_N)\big)=-\partial_x\phi(x_N,t_N)+O(N^{-1})\\
&N\big(\phi(x_N+1/N,t_N)-\phi(x_N,t_N)\big)=\partial_x\phi(x_N,t_N)+O(N^{-1})\;.
\end{align*}
Recall that $r=\frac{1+\lambda}{1-\lambda}$ and that $\log(r)\sim 2\lambda$ as $N\to\infty$. Then a straightforward Taylor expansion applied to the exponential terms of $H^N_X$ yield the desired convergence.\\
Regarding $H_M^N$, the only difference comes from the $\mu_k$'s. As $x_N\ge k_0/N$, we can bound $\mu_k\leqslant \mu_{k_0}$ and we note that $\mu_{k_0}\ll \lambda^2$ so that this term has negligible contributions and some Taylor expansions yield the desired bound.
\end{proof}

\begin{rem}
	This result is no longer true when the asymmetry parameter $\lambda > 0$ is independent of $N$. In that setting, one gets in the limit for $H^N_X$
	$$ \frac{1}{\lambda\log(r)}\left(\frac{1+\lambda}{2}r^{-\partial_x \phi}+\frac{1-\lambda}{2}r^{\partial_x \phi}-1\right)\;,$$
	and for $H^N_M$ 
	$$ \frac{1}{\lambda\log(r)}\log\left(\frac{1+\lambda}{2}r^{-\partial_x \phi}+\frac{1-\lambda}{2}r^{\partial_x \phi}\right)\;.$$
	This is the reason why we are not able to identify the hydrodynamic limit (under the transformation $T$) in that setting, and thus rely on the simpler hydrodynamic limit stated in Proposition \ref{Prop:HydroNaive}.
\end{rem}

\subsection{Convergence of the systems of ODEs}\label{Subsec:CVODE}

We need an a priori estimate on the solutions of the systems of ODEs, the proof of which is postponed to the end of the subsection.

\begin{prop}\label{Prop:Apriori}
	There exists some sequence $\eps_N$ going to $0$ as $N\to\infty$ such that
$$ 0 \le f_X^N(x,t) \le 1-x\;,\quad 0 \le f_M^N(x,t) \le (1+\eps_N)(1-x)\;,\quad (x,t) \in \Omega_N\;,$$
and
$$ f_X^N(0,t) , f_M^N(0,t) \ge (1-\eps_N) \min\big(\frac{t}{4},1\big) \;,\quad t\ge 0\;.$$
\end{prop}

The boundedness of these sequences allows us to introduce the following functions on $[0,1]\times\R^+$
\begin{align*}
\overline{f}_X(x,t)&:=\limsup_{(y_N,s_N)\to (x,t)} f_X^N(y_N,s_N)\;,\quad \underline{f}_X(x,t)&:=\liminf_{(y_N,s_N)\to (x,t)} f_X^N(y_N,s_N)\;,\\
\overline{f}_M(x,t)&:=\limsup_{(y_N,s_N)\to (x,t)} f_M^N(y_N,s_N)\;,\quad \underline{f}_M(x,t)&:=\liminf_{(y_N,s_N)\to (x,t)}f_M^N(y_N,s_N)\;,
\end{align*}
where the $\liminf$ and $\limsup$ are taken over all sequences of points $(y_N,s_N)\in\Omega_N$ that converge to $(x,t)$.

\begin{prop}\label{Prop:CVSchemesbar}
The functions $\overline{f}_X$ and $\overline{f}_M$ are viscosity \textbf{sub}-solution of \eqref{Eq:HJ}, while the functions $\underline{f}_X$ and $\underline{f}_M$ are viscosity \textbf{super}-solution of \eqref{Eq:HJ}.
\end{prop}

We emphasize that the $\limsup$ are \emph{sub}-solutions, while the $\liminf$ are \emph{super}-solutions. We will restrict ourselves to proving that $\overline{f}_X$ and $\overline{f}_M$ are sub-solutions. A simple adaptation of the computations allows to prove that $\underline{f}_X$ and $\underline{f}_M$ are super-solutions. Moreover, the arguments being similar for $f_X^N$ and $f_M^N$, we will write $f$ to denote either $f_X^N$ or $f_M^N$ and $H^N$ to denote either $H_X^N$ or $H_M^N$.

\begin{proof}[Proof of Proposition \ref{Prop:CVSchemesbar}]
Fix $(x,t)\in [0,1]\times\R_+ \backslash \{(0,0)\}$. Our aim is to show that $\overline{f}$ satisfies the viscosity inequality at $(x,t)$ if $(x,t) \notin \partial \Omega$, or satisfies either the viscosity inequality at $(x,t)$ or the boundary inequality at $(x,t)$ if $(x,t) \in \partial\Omega$. Let $\phi$ be $\mathcal{C}^\infty$ and such that $\overline{f} -\phi$ admits a maximum at $(x,t)$ over $V \cap ([0,1]\times\R_+)$ for some neighbourhood $V$ of $(x,t)$. By definition of $\overline{f}$, there is a sequence $(N_{n})_n$ and $(y_n,s_n)_n$ converging to $(x,t)$ such that $\lim_{n\to\infty} f^{N_n}(y_n,s_n)=\overline{f}(x,t)$. For simplicity, we will write $f^n$ instead of $f^{N_n}$. Let $V_n$ be an open ball such that: (1) $V_n \subset V$,  (2) $V_n$ contains a ball centered at $(x,t)$ whose radius vanishes with $n$, (3) $V_n$ contains $(y_n,s_n)$. (For $n$ large enough, these constraints are compatible). Let us now pick a sequence $(x_n,t_n)_n$ satisfying for every $n$
$$(x_n,t_n)\in \text{arg}\max\left\{(f^n-\phi)(z,r),(z,r)\in V_n\cap\Omega_n\right\}\;.$$
Then we claim that (recall the set $\partial \Omega$ introduced right below \eqref{Eq:HJ}):
\begin{enumerate}[label=(\roman*)]
\item $(x_n,t_n)\too(x,t)$
\item $f^n(x_n,t_n)\too\overline{f}(x,t)$
\item If $(x,t) \in \partial \Omega$ and if $\overline{f}(x,t)>\text{BC}(x,t)$ then for all $n$ large enough $(x_n,t_n)\in(0,1)\times\R_+^*$.
\end{enumerate}
\medskip

Proof of (i). This is an immediate consequence of $\cap_n V_n = \{(x,t)\}$.

\medskip

Proof of (ii). We have
$$\limsup_{n\to\infty}f^n(x_n,t_n)\leqslant\overline{f}(x,t)$$
and
$$\liminf_{n\to\infty} f^n(x_n,t_n)-\phi(x_n,t_n)\geqslant\liminf_{n\to\infty}f^n(y_n,s_n)-\phi(y_n,s_n)=\overline{f}(x,t)-\phi(x,t).$$
We deduce that
$$\liminf_{n\to\infty} f^n(x_n,t_n)\geqslant\liminf_{n\to\infty}f^n(y_n,s_n)=\overline{f}(x,t)\;,$$
and this concludes the proof of (ii).

\medskip

Proof of (iii). Assume that $(x,t) \in \partial \Omega$, and that there is a sub-sequence $(x_{n_k},t_{n_k})$ that lies in $\partial\Omega$. By (i) and (ii), we deduce that $(x_{n_k},t_{n_k})\too(x,t)$ and $f^{n_k}(x_{n_k},t_{n_k})\too\overline{f}(x,t)$. Moreover, by definition of the systems of ODEs we have $f^{n_k}(x_{n_k},t_{n_k})=$BC$(x_{n_k},t_{n_k})$. The continuity (recall that $(0,0) \notin \partial\Omega$) of BC on $\partial\Omega$ ensures that $\overline{f}(x,t)=$BC$(x,t)$. We deduce (iii).

\medskip

Given the claims above, we conclude the proof as follows. If $(x,t)$ lies on $\partial \Omega$ and $\overline{f}(x,t) \leqslant $ BC$(x,t)$ then we are done. Now assume either that $(x,t)$ lies on $\partial \Omega$ but $\overline{f}(x,t) > $ BC$(x,t)$, or that $(x,t)$ does not lie on $\partial \Omega$. For all $n$ large enough, it holds $(x_n,t_n) \in (0,1)\times\R_+^*$. Let us set $d_n=f^n(x_n,t_n)-\phi(x_n,t_n)$. Recall $(x_n,t_n)$ is an argument of the maximum of $f^n-\phi$, so we have $\partial_t f^n(x_n,t_n)=\partial_t\phi(x_n,t_n)$.
 We compute (we omit writing the fourth argument $Nx_n$ of $H^{N_n}$ in the particular case where $H^{N_n} = H^{N_n}_M$):
\begin{align*}
&\partial_t\phi(x_n,t_n)=\partial_t f^n(x_n,t_n)\\
&=-H^{N_n}(f^n(x_n-1/N_n,t_n),f^n(x_n,t_n),f^n(x_n+1/N_n,t_n))\\
&\leqslant -H^{N_n}(\phi(x_n-1/N_n,t_n)+d_n,\phi(x_n,t_n)+d_n,\phi(x_n+1/N_n,t_n)+d_n)\\
&\leqslant -H^{N_n}(\phi(x_n-1/N_n,t_n),\phi(x_n,t_n),\phi(x_n+1/N_n,t_n))
\end{align*}
where we have used $f^n(x_n,t_n)=\phi(x_n,t_n)+d_n$ at the third line and Lemma \ref{increasing} at the third and fourth lines. 
By Proposition \ref{Prop:PDE}, taking the limit $n\to\infty$ we obtain
$$ \partial_t \phi(x,t) + \partial_x \phi(x,t) + (\partial_x \phi(x,t))^2 \le 0\;,$$
as required.
\end{proof}

We now conclude the proof of the convergence of the systems of ODEs.
\begin{proof}[Proof of Proposition \ref{Prop:CVschemes}]
First of all, the definitions of $\underline{f}, \overline{f}$ as $\liminf, \limsup$ imply that $\underline{f}$ is l.s.c.~and $\overline{f}$ is u.s.c. In addition, the monotonicity of $x\mapsto f^N(x,t)$ and $t\mapsto f^N(x,t)$ for all $N\geqslant 1$ yields that $x\mapsto \underline{f}(x,t)$, $x\mapsto \overline{f}(x,t)$, $t\mapsto \underline{f}(x,t)$, $t\mapsto \overline{f}(x,t)$ are monotonous too.\\
From Proposition \ref{Prop:Apriori}, we deduce that
$$ 0 \leqslant \overline{f}(x,t), \underline{f}(x,t) \leqslant 1-x\;,\quad (x,t) \in [0,1]\times \R_+\;.$$
Moreover, from the very definitions of these quantities, we obviously have $\underline{f} \leqslant \overline{f}$. The previous proposition showed that $\underline{f}$ is a super-solution of \eqref{Eq:HJ} while $\overline{f}$ is a sub-solution. By Lemma \ref{Lemma:u} and the inequality $\underline{f} \le \overline{f}$, we deduce that for all $x\in (0,1]$ and for all $t\in [0,4)$
\begin{equation}\label{Eq:underovert4}
	\underline{f}(x,t) \leqslant \overline{f}(x,t)\leqslant \frac{t}{4}\;.
\end{equation}
By Proposition \ref{Prop:Apriori}, we obtain for all $t\in [0,4)$
$$ \min(\frac{t}{4},1) \leqslant \underline{f}(0,t) \leqslant \overline{f}(0,t)\;.$$
This, combined with \eqref{Eq:underovert4} and the lower-semicontinuity of $\underline{f}$, yields
$$ \underline{f}(0,t) = \frac{t}{4}\;,\quad \forall t\in [0,4)\;.$$ 
Combined with the monotonicity in time of this function, we deduce that
$$ \underline{f}(0,t) = \min(\frac{t}{4},1)\;,\quad \forall t \ge 0\;.$$ 
We would like to apply the uniqueness result stated in Proposition \ref{Th:Uniq} and deduce that $\underline{f} \geqslant \overline{f}$, as this would imply equality of these two quantities. Unfortunately, this cannot be true since $\overline{f}(0,t) = 1$ for all $t\ge 0$, while $\underline{f}(0,t) = \min(\frac{t}{4},1)$.\\
Actually, Proposition \ref{Th:Uniq} cannot be applied because $x\mapsto \overline{f}(x,t)$ is \emph{not} continuous at $0$ for $t< 4$. To circumvent this issue, we introduce
$$\overline{\overline{f}}(x,t)=\begin{cases} \overline{f}(x,t)&\text{if } x>0\\
\lim_{y\downarrow 0}\overline{f}(y,t)&\text{if } x=0
\end{cases}$$
(This is well-defined by monotonicity). This function is still a sub-solution as it coincides with a sub-solution for $x \in (0,1]$, and satisfies the boundary condition $\overline{\overline{f}}(0,t) \le 1$ at $x=0$. We can now apply Proposition \ref{Th:Uniq} and deduce that $\underline{f} \ge \overline{\overline{f}}$. On the other hand, from their very definitions we see that necessarily $\underline{f} \le \overline{\overline{f}}$. This ensures that $\underline{f} = \overline{\overline{f}} = S$.\\
We have therefore shown that $f^N_X(x_N,t), f^N_M(x_N,t)$ converge to $S(x,t)$ for all $(x,t)$. We remark this convergence is locally uniform because $S$ is continuous and because of the definitions of $\underline{f}$ and $\overline{f}$. This is exactly the desired convergence for $T(\E[X^{\max}])$.
Regarding $\E[T^M]$, we need an additional argument to conclude. Fix $t_0> 0$. By Lemma \ref{Lemma:MY}, we know that
$$ \P\big(\forall k, \forall t\in [0,t_0 N/\lambda],\quad T^M_k(tN/\lambda) \geqslant T^X_k(tN/\lambda) \big) \too 1\;,\quad N\to\infty\;.$$
Since $T$ is valued in $[0,1]$, we thus deduce that for all $\eps > 0$, we have for all $N$ large enough, all $k$ and all $t\le t_0$
\begin{equation}\label{Eq:TYM2}
f_X^N(k/N,t) = \E[T^X_k(tN/\lambda)] \leqslant \E[T^M_k(tN/\lambda)] + \eps  \leqslant f_M^N(k/N,t) + \eps\;,
\end{equation}
which suffices to conclude.
\end{proof}

The remaining of this subsection is devoted to the proof of Proposition \ref{Prop:Apriori}. We introduce the notions of sub and super solutions of our systems of ODEs.

\begin{defn}\label{defSol}
We say that $u$ is a sub-solution (respectively super-solution) of the system associated to $H_X^N$ if $t\mapsto u(x,t)$ is differentiable and if for all $(x,t)\in \Omega_N$ it holds:
$$\partial_t u(x,t) + H_X^N(u(x-1/N,t),u(x,t),u(x+1/N,t)) \leqslant 0\;,$$
(resp.~$\geqslant$) and
\begin{align*}
u(0,t) \leqslant 1\;,\quad u(1,t)\leqslant 0\;,\quad u(x,0) \leqslant 0
\end{align*}
(resp.~$\geqslant$). We adopt similar definitions for the sub- and super-solutions of the system associated to $H_M^N$.
\end{defn}

\begin{prop}\label{Prop:CompareScheme}
If $u$ is a sub-solution and $v$ a super-solution of either system, then $u\leqslant v$.
\end{prop}
\begin{proof}
Denote $$t_0=\inf\{t\geqslant 0,\exists k\in\llbracket 1,N-1\rrbracket,\ v(k/N,t)<u(k/N,t)\}\;.$$ We suppose by contradiction that $t_0<\infty$. Using the continuity in $t$ of $u$ and $v$, we have for all $k\in\llbracket 1,N-1\rrbracket$, the inequality $v(k/N,t_0)\geqslant u(k/N,t_0)$ and for at least one $k\in\llbracket 1,N-1\rrbracket$, the equality $v(k/N,t_0)= u(k/N,t_0)$. For any $k$ such that $v(k/N,t_0)= u(k/N,t_0)$, the monotonicity stated in Lemma \ref{increasing} yields
$$\partial_t v(k/N,t_0)\geqslant \partial_t u(k/N,t_0)\;,$$
and therefore $v(k/N,t) \geqslant u(k/N,t)$ for all $t$ in some neighborhood of $t_0$. On the other hand, for any $k$ such that $v(k/N,t_0)> u(k/N,t_0)$, the continuity in $t$ ensures that this also remains true in a neighborhood of $t_0$. This raises a contradiction with the finiteness of $t_0$, and therefore $t_0=\infty$.
\end{proof}

We can now proceed with the proof of our a priori bounds.
\begin{proof}[Proof of Proposition \ref{Prop:Apriori}]
It is elementary to check that $u\equiv 0$ is a sub-solution of both schemes. A simple computation shows that the function $v_X(x,t) = 1-x$ is a super-solution for $H_X^N$. On the other hand, identifying a super-solution for $H_M^N$ is less immediate. Recall the definition of $k_0$ from the beginning of Subsection \ref{Subsec:M}. Define for every $(x,t) \in \Omega_N$
$$v_M(x,t)=\begin{cases}
1 &\text{if }x < k_0/N\\
c_N(1-x) &\text{if }x\ge k_0/N
\end{cases}\;,\quad \text{with }c_N := \frac1{1- \frac{k_0}{N}}\;.$$
Let us check that $v_M$ is a super-solution for $H_M^N$. The boundary inequalities associated to the scheme $H_M^N$ are trivially satisfied. Regarding the evolution in time, we distinguish three cases according to the relative values of $xN$ and $k_0$.\\
The function $v_M$ is constant to the left of $k_0/N$. Since $H_M^N(u,u,u,k) = 0$, we deduce that whenever $0 < xN < k_0$ if we set $k=xN$ we have $\partial_t v_M(x,t) = -H_M^N(v_M(x-1/N,t) , v_M(x,t) , v_M(x-1/N,t),k)$.\\
Now assume that $k_0 < xN < N$ and set $k=xN$. Since $\mu_k \leqslant \mu_{k_0}$ and $r^{-c_N} - r^{c_N} < 0$, we get
\begin{align*}
&H_M^N(c_N(1-x+1/N),c_N(1-x),c_N(1-x-1/N),k)\\
&= \frac{1}{\lambda}\log_r\left(\frac{1+\lambda+\mu_k}{2}r^{-c_N}+\frac{1-\lambda-\mu_k}{2}r^{c_N}\right)\\
&\geqslant \frac{1}{\lambda}\log_r\left(\frac{1+\lambda+\mu_{k_0}}{2}r^{-c_N}+\frac{1-\lambda-\mu_{k_0}}{2}r^{c_N}\right)\;.
\end{align*}
Let us check that $g(c_N) \geqslant 1$ where
$$ g(x) := \frac{1+\lambda+\mu_{k_0}}{2}r^{-x}+\frac{1-\lambda-\mu_{k_0}}{2}r^{x}\;.$$
We have $g(c'_N) = 1$ for $c'_N := \log_r( (1+\lambda+\mu_{k_0}) / (1-\lambda-\mu_{k_0}))$. Furthermore, a computation shows that $g$ is non-decreasing on $[c'_N,\infty)$. Finally, it can be checked that $\mu_{k_0}/\lambda \ll k_0/N$ so that a Taylor expansion shows that $c_N \ge c'_N$ holds for all $N$ large enough.\\
We thus deduce that
$$ -H_M^N(c_N(1-x+1/N),c_N(1-x),c_N(1-x-1/N),k) \leqslant 0 = \partial_t v_M(x,t)\;,$$
as required.\\
Finally if $x = k_0/N$, then the previous computation combined with the monotonicity of $H_M^N$ stated in Lemma \ref{increasing} yields the desired result. This shows that $v_M$ is indeed a super-solution for $H^N_M$.\\
Since $f_X^N$ (resp.~$f_M^N$) is both a sub- and a super-solution for $H_X^N$ (resp.~$H_M^N$), we deduce from Proposition \ref{Prop:CompareScheme} that $0\leqslant f_X^N(x,t) \leqslant 1-x$ and $0 \leqslant f_M^N(x,t) \leqslant c_N(1-x)$, as desired.\\

We turn to the a priori lower bound at $x=0$. We define $f(x,t)=\frac{-x^2}{16-4t}-\frac{x}{2}+\frac{t}{4}$. One can check $f$ is a $C^\infty$ solution of the equation $$\partial_t f+\partial_x f+(\partial_x f)^2=0$$
on $[0,1]\times[0,4)$. Fix $T<4$. Inspecting the proof of Proposition \ref{Prop:PDE}, one can check that there exists $C>0$ such that
$$\Big| H_X^N(f(\frac{k-1}{N},t),f(\frac{k}{N},t),f(\frac{k+1}{N},t))-\partial_x f(\frac{k}{N},t)-(\partial_x f(\frac{k}{N},t))^2 \Big|$$
is bounded by $C\max(\lambda, 1/(N\lambda))$ uniformly over all $k\in \{1,\ldots,N-1\}$ and all $t\in [0,T]$. The same holds for $H_M^N$ but only for $k\geqslant k_0$. This being given, we set $u^N(k/N,t) := f(k/N,t) - C\max(\lambda, 1/(N\lambda)) t$ for all $k$ and $t\in [0,4)$, and we claim this is a sub-solution for both systems associated to $H^N_X$ and $H^N_M$. It is immediate for $H^N_X$. Regarding $H^N_M$, it is immediate that
$$\partial_t u^N(k/N,t)\leqslant -H_M^N(u^N(\frac{k-1}{N},t),u^N(\frac{k}{N},t),u^N(\frac{k+1}{N},t),k)$$
for $k\geqslant k_0$. On the other hand for $k< k_0$ we have $\mu_k = 1-\lambda$ and thus $$H_M^N(f(\frac{k-1}{N},t),f(\frac{k}{N},t),f(\frac{k+1}{N},t),k)=\frac{1}{\lambda}N(f(\frac{k}{N},t)-f(\frac{k-1}{N},t))$$and the right hand side goes to $-\infty$ as $N\to\infty$, while $\partial_t f(k/N,t)-C\max(\lambda, 1/(N\lambda))$ remains bounded. This ensures that for $N$ large enough
$$\partial_t u^N(k/N,t)\leqslant -H_M^N(u^N(\frac{k-1}{N},t),u^N(\frac{k}{N},t),u^N(\frac{k+1}{N},t),k)\;.$$
 The claim follows. Applying Proposition \ref{Prop:CompareScheme} (on the finite interval of time $[0,T]$) and passing to the limit on $N$, we deduce that 
$$\underline{f}_X(0,t)\geqslant \frac{t}{4}\text{ and }\underline{f}_M(0,t)\geqslant \frac{t}{4}\;.$$Since $T$ can be taken as close to $4$ as desired, this concludes the proof.
\end{proof}

\subsection{Proof of the first hydrodynamic limit}\label{Subsec:Naive}

\begin{proof}[Proof of Proposition \ref{Prop:HydroNaive} and Theorem \ref{Th:2}]
The upper bound of Theorem \ref{Th:2} was proved in Section \ref{Sec:Upper}. The lower bound follows from concentration estimates on the invariant measure, similar to those presented in the proof of Theorem \ref{Th:Main}, combined with the hydrodynamic limit of Proposition \ref{Prop:HydroNaive}, that we now prove.\\
Assume that  for all $x \in (0,1)$ and $t > 0$ such that $x\ne t$, as $N\to\infty$ the following convergence holds
\begin{equation}\label{Eq:CVNaive}
\frac{1}{N}\E[X^{\max}_{\lfloor xN\rfloor}(tN/\lambda)] \to \mathbf{1}_{[t,1]}(x)\;.
\end{equation}
Since $\frac{1}{N} X^{\max}$ takes values in $[0,1]$, this immediately implies the statement of the proposition. We are left with the proof of the convergence.\\
The map $u^N(k/N,t) :=\frac{1}{N}\E(X^{\max}_{k}(tN/\lambda)$ is solution of the following system of ODEs:
\begin{align*}
\partial_t u(x,t)+H^N(u(x-1/N,t),u(x,t),u(x+1/N,t)) = 0\;,\quad (x,t)\in\Omega_N\;,\\
u(0,t)=0\;,\quad u(1,t)=1\;,\quad u(k/N,0)=1\;,
\end{align*}
with $H^N(x,y,z)=-\frac{N}{\lambda}\left(\frac{1+\lambda}{2}(x-y)+\frac{1-\lambda}{2}(z-y)\right)$. The function $H^N$ is non-increasing in its first and third variables so we can introduce the notion of sub- and super-solution of this system as in the previous subsection. The comparison of Lemma \ref{Prop:CompareScheme} between sub- and super-solutions remains in force in this context. We will exhibit a sub-solution $f_-^N$ such that for $x>t$, we have $f_-^N(\lfloor xN\rfloor,t)\too 1$ as $N\to\infty$ and a super-solution $f_-^N$ such that for $x<t$, we have $f_-^N(\lfloor xN\rfloor,t)\too 0$ as $N\to\infty$. By comparison, this will imply that $f_-^N\leqslant u^N \leqslant f_+^N$ and the desired property will thus follow.\\
We define for $(k,t)\in\Omega_N$:
\begin{align*}
f_-^N(k,t)&=1-(N\lambda)^{2/3}\left(t-\frac{k}{N}+\frac{1}{(N\lambda)^{1/3}}\right)_+^2-\frac{t}{(N\lambda)^{1/3}}\;,\\
f_+^N(k,t)&=(N\lambda)^{2/3}\left(\frac{k}{N}+\frac{1}{(N\lambda)^{1/3}}-t\right)_+^2+\frac{t}{(N\lambda)^{1/3}}\; .
\end{align*}
We restrict ourselves to showing that $f_-^N$ is a sub-solution, since the arguments to show that $f_+^N$ is a super-solution are quite similar. It is easy to check that $f_-^N(0,t)\leqslant 0,\ f_-^N(k,0)\leqslant 1$ and $f_-^N(N,t)\leqslant 1$. Regarding the evolution in time, we need to argue differently according to the relative values of $t$ and $k$. To alleviate the notations, let us write $H$ for $H^N(f_-^N(k-1,t),f_-^N(k,t),f_-^N(k+1,t))$.\\
If $t\leqslant \frac{k-1}{N}-\frac{1}{(N\lambda)^{1/3}}$, we have $H=0$ so $\partial_t f_-^N(k,t)\leqslant -H$.\\
If $t\geqslant \frac{k+1}{N}-\frac{1}{(N\lambda)^{1/3}}$ a computation shows that
\begin{align*}
H&=(N\lambda)^{2/3}2\left(t-\frac{k}{N}+\frac{1}{(N\lambda)^{1/3}}\right)+\frac{1}{(N\lambda)^{1/3}} =- \partial_t f_-^N(k,t)\;.
\end{align*}
If $t\in [\frac{k}{N}-\frac{1}{(N\lambda)^{1/3}},\frac{k+1}{N}-\frac{1}{(N\lambda)^{1/3}}]$, as $H^N$ is non-increasing in its third variable we can replace $f_-^N(k+1,t)$ by $1-(N\lambda)^{2/3}\left(t-\frac{k+1}{N}+\frac{1}{(N\lambda)^{1/3}}\right)^2-\frac{t}{(N\lambda)^{1/3}}$ and the previous computation shows that $\partial_t f_-^N(k,t)\leqslant -H$.\\
If $t\in [\frac{k-1}{N}-\frac{1}{(N\lambda)^{1/3}},\frac{k}{N}-\frac{1}{(N\lambda)^{1/3}}]$, we compute 
\begin{align*}
-H&=-(N\lambda)^{2/3}\frac{N}{\lambda}\frac{1+\lambda}{2}\left(t-\frac{k-1}{N}+\frac{1}{(N\lambda)^{1/3}}\right)^2\\
&\geqslant -\frac{1+\lambda}{2}\frac{1}{(N\lambda)^{1/3}}\\
&\geqslant -\frac{1}{(N\lambda)^{1/3}}=\partial_t f_-^N(k,t)\; .
\end{align*}
\end{proof}

\bibliographystyle{Martin}
\bibliography{library}

\begin{thebibliography}{BRAS06}
\expandafter\ifx\csname url\endcsname\relax
  \def\url#1{\texttt{#1}}\fi
\expandafter\ifx\csname urlprefix\endcsname\relax\def\urlprefix{URL }\fi
\expandafter\ifx\csname href\endcsname\relax
  \def\href#1#2{#2}\fi
\expandafter\ifx\csname burlalt\endcsname\relax
  \def\burlalt#1#2{\href{#2}{\texttt{#1}}}\fi

\bibitem[AD86]{AldDia}
\textsc{D.~Aldous} and \textsc{P.~Diaconis}.
\newblock Shuffling cards and stopping times.
\newblock \emph{Amer. Math. Monthly} \textbf{93}, no.~5, (1986), 333--348.

\bibitem[BHP21]{BarHogPar}
\textsc{G.~Barrera}, \textsc{M.~A. H\"{o}gele}, and \textsc{J.~C. Pardo}.
\newblock The cutoff phenomenon in total variation for nonlinear {L}angevin
  systems with small layered stable noise.
\newblock \emph{Electron. J. Probab.} \textbf{26}, (2021), Paper No. 119, 76.
\newblock \burlalt{doi:10.1214/21-ejp685}{http://dx.doi.org/10.1214/21-ejp685}.

\bibitem[BJ16]{Barrera}
\textsc{G.~Barrera} and \textsc{M.~Jara}.
\newblock Abrupt convergence for stochastic small perturbations of one
  dimensional dynamical systems.
\newblock \emph{J. Stat. Phys.} \textbf{163}, no.~1, (2016), 113--138.
\newblock
  \burlalt{doi:10.1007/s10955-016-1468-1}{http://dx.doi.org/10.1007/s10955-016-1468-1}.

\bibitem[BRAS06]{Balazs}
\textsc{M.~Bal\'{a}zs}, \textsc{F.~Rassoul-Agha}, and
  \textsc{T.~Sepp\"{a}l\"{a}inen}.
\newblock The random average process and random walk in a space-time random
  environment in one dimension.
\newblock \emph{Comm. Math. Phys.} \textbf{266}, no.~2, (2006), 499--545.
\newblock
  \burlalt{doi:10.1007/s00220-006-0036-y}{http://dx.doi.org/10.1007/s00220-006-0036-y}.

\bibitem[CIL92]{CIL}
\textsc{M.~G. Crandall}, \textsc{H.~Ishii}, and \textsc{P.-L. Lions}.
\newblock User’s guide to viscosity solutions of second order partial
  differential equations.
\newblock \emph{Bulletin of the American Mathematical Society} \textbf{27},
  no.~1, (1992), 1–68.
\newblock
  \burlalt{doi:10.1090/s0273-0979-1992-00266-5}{http://dx.doi.org/10.1090/s0273-0979-1992-00266-5}.

\bibitem[CL83]{CrandallLions}
\textsc{M.~G. Crandall} and \textsc{P.-L. Lions}.
\newblock Viscosity solutions of {H}amilton-{J}acobi equations.
\newblock \emph{Trans. Amer. Math. Soc.} \textbf{277}, no.~1, (1983), 1--42.
\newblock \burlalt{doi:10.2307/1999343}{http://dx.doi.org/10.2307/1999343}.

\bibitem[CLL20]{CLL}
\textsc{P.~Caputo}, \textsc{C.~Labb\'{e}}, and \textsc{H.~Lacoin}.
\newblock Mixing time of the adjacent walk on the simplex.
\newblock \emph{Ann. Probab.} \textbf{48}, no.~5, (2020), 2449--2493.
\newblock
  \burlalt{doi:10.1214/20-AOP1428}{http://dx.doi.org/10.1214/20-AOP1428}.

\bibitem[Dia96]{diaconis1996cutoff}
\textsc{P.~Diaconis}.
\newblock The cutoff phenomenon in finite markov chains.
\newblock \emph{Proceedings of the National Academy of Sciences} \textbf{93},
  no.~4, (1996), 1659--1664.

\bibitem[DS81]{DiaSha}
\textsc{P.~Diaconis} and \textsc{M.~Shahshahani}.
\newblock Generating a random permutation with random transpositions.
\newblock \emph{Z. Wahrsch. Verw. Gebiete} \textbf{57}, no.~2, (1981),
  159--179.

\bibitem[HJ17]{HoughJiang17}
\textsc{B.~Hough} and \textsc{Y.~Jiang}.
\newblock Cut-off phenomenon in the uniform plane {K}ac walk.
\newblock \emph{Ann. Probab.} \textbf{45}, no.~4, (2017), 2248--2308.
\newblock
  \burlalt{doi:10.1214/16-AOP1111}{http://dx.doi.org/10.1214/16-AOP1111}.

\bibitem[Lac05]{Lachaud}
\textsc{B.~Lachaud}.
\newblock Cut-off and hitting times of a sample of {O}rnstein-{U}hlenbeck
  processes and its average.
\newblock \emph{J. Appl. Probab.} \textbf{42}, no.~4, (2005), 1069--1080.
\newblock
  \burlalt{doi:10.1239/jap/1134587817}{http://dx.doi.org/10.1239/jap/1134587817}.

\bibitem[Lio82]{Lions}
\textsc{P.-L. Lions}.
\newblock \emph{Generalized solutions of {H}amilton-{J}acobi equations},
  vol.~69 of \emph{Research Notes in Mathematics}.
\newblock Pitman (Advanced Publishing Program), Boston, Mass.-London, 1982.

\bibitem[LL19]{LLASEP}
\textsc{C.~Labb\'{e}} and \textsc{H.~Lacoin}.
\newblock Cutoff phenomenon for the asymmetric simple exclusion process and the
  biased card shuffling.
\newblock \emph{Ann. Probab.} \textbf{47}, no.~3, (2019), 1541--1586.
\newblock
  \burlalt{doi:10.1214/18-AOP1290}{http://dx.doi.org/10.1214/18-AOP1290}.

\bibitem[LL20]{LLWASEP}
\textsc{C.~Labb\'{e}} and \textsc{H.~Lacoin}.
\newblock Mixing time and cutoff for the weakly asymmetric simple exclusion
  process.
\newblock \emph{Ann. Appl. Probab.} \textbf{30}, no.~4, (2020), 1847--1883.
\newblock
  \burlalt{doi:10.1214/19-AAP1545}{http://dx.doi.org/10.1214/19-AAP1545}.

\bibitem[LPW17]{LevPerWil}
\textsc{D.~A. Levin}, \textsc{Y.~Peres}, and \textsc{E.~L. Wilmer}.
\newblock \emph{Markov chains and mixing times}.
\newblock American Mathematical Society, Providence, RI, 2017.
\newblock Second edition of [ MR2466937], With a chapter on ``Coupling from the
  past'' by James G. Propp and David B. Wilson.

\bibitem[LT19]{Toninelli}
\textsc{M.~Legras} and \textsc{F.~L. Toninelli}.
\newblock Hydrodynamic limit and viscosity solutions for a two-dimensional
  growth process in the anisotropic kpz class.
\newblock \emph{Communications on Pure and Applied Mathematics} \textbf{72},
  no.~3, (2019), 620--666.
\newblock
  \burlalt{doi:https://doi.org/10.1002/cpa.21796}{http://dx.doi.org/https://doi.org/10.1002/cpa.21796}.

\bibitem[Luk55]{lukacs1955}
\textsc{E.~Lukacs}.
\newblock A characterization of the gamma distribution.
\newblock \emph{The Annals of Mathematical Statistics} \textbf{26}, no.~2,
  (1955), 319--324.

\bibitem[M\'14]{Meliot}
\textsc{P.-L. M\'{e}liot}.
\newblock The cut-off phenomenon for {B}rownian motions on compact symmetric
  spaces.
\newblock \emph{Potential Anal.} \textbf{40}, no.~4, (2014), 427--509.
\newblock
  \burlalt{doi:10.1007/s11118-013-9356-7}{http://dx.doi.org/10.1007/s11118-013-9356-7}.

\bibitem[Pet22]{Enguerand}
\textsc{E.~Petit}.
\newblock Cutoff for the weakly asymmetric adjacent walk on the simplex.
\newblock \emph{in preparation} (2022+).

\bibitem[Rez01]{Reza2}
\textsc{F.~Rezakhanlou}.
\newblock {Continuum Limit for Some Growth Models II}.
\newblock \emph{The Annals of Probability} \textbf{29}, no.~3, (2001), 1329 --
  1372.
\newblock
  \burlalt{doi:10.1214/aop/1015345605}{http://dx.doi.org/10.1214/aop/1015345605}.

\bibitem[Rez02]{Reza1}
\textsc{F.~Rezakhanlou}.
\newblock Continuum limit for some growth models.
\newblock \emph{Stochastic Processes and their Applications} \textbf{101},
  no.~1, (2002), 1--41.
\newblock
  \burlalt{doi:https://doi.org/10.1016/S0304-4149(02)00100-X}{http://dx.doi.org/https://doi.org/10.1016/S0304-4149(02)00100-X}.

\bibitem[RW05]{RW05}
\textsc{D.~Randall} and \textsc{P.~Winkler}.
\newblock Mixing points on an interval.
\newblock In \emph{Proceedings of the Second Workshop on Analytic Algorithms
  and Combinatorics, Vancouver, 2005},  216--221. 2005.

\bibitem[Sep99]{Seppa}
\textsc{T.~Seppäläinen}.
\newblock {Existence of Hydrodynamics for the Totally Asymmetric Simple
  K-Exclusion Process}.
\newblock \emph{The Annals of Probability} \textbf{27}, no.~1, (1999), 361 --
  415.
\newblock
  \burlalt{doi:10.1214/aop/1022677266}{http://dx.doi.org/10.1214/aop/1022677266}.

\bibitem[{Zha}18]{Zhang}
\textsc{X.~{Zhang}}.
\newblock {The domino shuffling height process and its hydrodynamic limit}.
\newblock \emph{arXiv e-prints}  arXiv:1808.07409.
\newblock \burlalt{arXiv:1808.07409}{http://arxiv.org/abs/1808.07409}.

\end{thebibliography}

\end{document}